\documentclass[12pt,a4paper]{amsart}
\usepackage{amsmath,mathrsfs,amssymb,amsthm,amscd,url}
\usepackage[]{fontenc}
\usepackage{xy} 
\usepackage{enumerate}
\usepackage{hyperref}
\hypersetup{breaklinks=true}

\xyoption{all}

\numberwithin{equation}{section}
\newcommand{\nnpar}[1]{\vskip 10pt \noindent {\scshape #1}}\nopagebreak%
\theoremstyle{plain}
\newtheorem{theorem}[equation]{Theorem}
\newtheorem{corollary}[equation]{Corollary}
\newtheorem{proposition}[equation]{Proposition}
\newtheorem{lemma}[equation]{Lemma}

\theoremstyle{definition}
\newtheorem{definition}[equation]{Definition}
\newtheorem{notation}[equation]{Notation}

\newtheorem{remark}[equation]{Remark}

\DeclareMathOperator{\bdv}{b-div}

\DeclareMathOperator{\dv}{div}

\DeclareMathOperator{\dd}{d}

\DeclareMathOperator{\Vol}{Vol}

\DeclareMathOperator{\val}{val}

\DeclareMathOperator{\conv}{conv}
\DeclareMathOperator{\Res}{Res}

\DeclareMathOperator{\Bir}{\mathcal{BIR}}
\DeclareMathOperator{\QCa}{\QQ-Ca}
\DeclareMathOperator{\QWe}{\QQ-We}

\newcommand{\cusp}{\text{{\rm cusp}}}
\newcommand{\sing}{\text{{\rm sing}}}
\newcommand{\can}{\text{{\rm can}}}
\newcommand{\weak}{\text{{\rm weak}}}

\newcommand{\op}{\text{{\rm op}}}

\newcommand{\w}{\text{{\rm pre}}}

\newcommand{\ZZ}{{\mathbb Z}}
\newcommand{\RR}{{\mathbb R}}

\newcommand{\CC}{{\mathbb C}}
\newcommand{\QQ}{{\mathbb Q}}

\newcommand{\PP}{{\mathbb P}}

\def\?{\ ???\ \immediate\write16{}%
\immediate\write16{Warning: There was still a question mark . . . }%
\immediate\write16{}}

\begin{document}
\setcounter{tocdepth}{1}
\setcounter{section}{0}

\title{The singularities of the invariant metric on the line bundle of
Jacobi forms}

\author[Burgos Gil]{Jos\'e Ignacio Burgos Gil}
\address{Instituto de Ciencias Matem\'aticas (CSIC-UAM-UCM-UCM3).
  Calle Nicol\'as Ca\-bre\-ra~15, Campus UAB, Cantoblanco, 28049 Madrid,
  Spain.} 
\email{burgos@icmat.es}
\urladdr{\url{http://www.icmat.es/miembros/burgos/}}
\author[Kramer]{J\"urg Kramer}
\address{Humboldt-Universit\"at zu Berlin, Institut f\"ur Mathematik,
Unter den Linden 6, 
D-10099 Berlin, Germany.}
\email{kramer@math.hu-berlin.de}
\urladdr{\url{http://www.math.hu-berlin.de/~kramer/}}
\author[K\"uhn]{Ulf K\"uhn}
\address{Fachbereich Mathematik (AZ),
Universit\"at Hamburg,
Bundesstra\ss{}e 55,
D-20146 Hamburg, Germany.}
\email{kuehn@math.uni-hamburg.de}
\urladdr{\url{http://www.math.uni-hamburg.de/home/kuehn/}}

\thanks{Burgos Gil was partially supported by the MICINN research
  projects MTM2009-14163-C02-01 and MTM2010-17389. Kramer acknowledges
  support from the DFG Graduate School \emph{Berlin Mathematical
    School} and from the DFG International Research Training Group
  \emph{Moduli and Automorphic Forms.}}

\begin{abstract}
A theorem by Mumford implies that every automorphic line bundle on a
pure open Shimura variety, equipped with an invariant smooth metric,
can be uniquely extended as a line bundle on a toroidal
compactification of the variety, in such a way that the metric
acquires only logarithmic singularities. This result is the key of
being able to compute arithmetic intersection numbers from these line
bundles. 
Hence it is natural to ask whether Mumford's result remains valid for
line bundles on mixed Shimura varieties.

In this paper we examine the simplest case, namely the sheaf of
Jacobi forms on the universal elliptic curve. We will show that
Mumford's result cannot be extended directly to this case and that a new
interesting kind of singularities appears. 

By using the theory of b-divisors, we show that an analogue of Mumford's
extension theorem can be obtained. We also show that this extension is
meaningful because it satisfies Chern-Weil theory and a
Hilbert-Samuel type of formula.
\end{abstract}
\maketitle

\tableofcontents

\section{Introduction}
\label{sec:introduction}

In \cite{Faltings:EaVZ}, \cite{Faltings:ftavnf}, Faltings
introduced the notion of logarithmically singular metrics on a
projective variety defined over a number field and proved
that they satisfy a Northcott type property, namely that the set of
algebraic points not lying on 
the singular set of the metric with bounded height and degree, is
finite. A prominent example of logarithmically singular metric is the
Hodge bundle $\omega$ on a toroidal compactification of the moduli space of
principally polarized abelian varieties of dimension $g$ (with level
structure if you do not want to work with stacks) $\overline
{\mathscr{A}}_{g}$ equipped with the Petersson metric.  

On the other hand, Mumford \cite{Mumford:Hptncc} introduced the
concept of a good metric on a vector bundle, which is a class of
singular metrics. He showed that, even being singular, Chern-Weil
theory carries over to good metrics. He also proved that the invariant
metric on a fully decomposable automorphic vector bundle on a
toroidal compactification of the quotient of a hermitian symmetric
domain by an arithmetic group is a good metric. This fact allowed him
to extend Hirzebruch's proportionality principle to non-compact
varieties. 

The conclusion of the above facts is that the natural metrics that
appear when studying vector bundles on toroidal compactifications of
pure Shimura varieties are singular, but the singularities are mild
enough so we can use the metrics to study geometric and arithmetic
problems. 

In \cite{BurgosKramerKuehn:accavb}
 and \cite{BurgosKramerKuehn:cacg}, the authors developed a
general theory of arithmetic intersections with logarithmically
singular metrics  that has been extensively used to compute
arithmetic intersection numbers
\cite{BruinierBurgosKuehn:MR2322676, KudlaRappoportYang:mfscsc,
  BruinierYang:fhcmcdls, Howard:CMckrd,Howard:CMKRii,
  BruinierHowardYang:hkrd, BerndtKuehn:KGf22I,
  BerndtKuehn:KGf22II, Freixas:hmls}.  

It is natural to ask whether this theory of logarithmically singular
metrics can be extended to mixed Shimura varieties, to obtain
geometric and arithmetic information of them.

In this paper we examine the first example of a mixed Shimura variety,
namely the universal elliptic curve of full level $N$ over the
modular curve $E^{0}(N)\to Y(N)$. On it we
consider the line bundle of Jacobi forms equipped with the translation
invariant metric. 

It turns out that, in this case, a new kind of singularities
appears. These new singularities are concentrated in codimension
two. Therefore, if we remove a set of codimension two, we can extend
the line bundle of Jacobi forms to a line bundle with a good hermitian
metric on a partial compactification of $E^{0}(N)$. Since algebraic
line bundles can be uniquely extended along codimension two subsets,
we obtain a line with a singular metric on a compactification $E(N)$
of the universal elliptic curve.

It turns out that this naive approach is not a good idea. First, it is
not functorial. If we consider different toroidal compactifications of
$E^{0}(N)$, then the resulting extensions are not compatible. Second,
even if the characteristic forms associated with the metric are
locally integrable and define cohomology classes, they fail
to satisfy a Chern-Weil theory. The cohomology class of the the
characteristic form does not agree with the characteristic class of
the extended line bundle.

In this paper we propose a different approach to understand the
extension of the line bundle of Jacobi forms to a compactification of
the universal elliptic curve. The ``right'' extension
is not a line bundle, but a b-$\QQ$-Cartier divisor. That is, a limit
of different Cartier divisors with rational coefficients on all
possible toroidal compactifications of $E^{0}(N)$. Defined in this
way, the extension is obviously functorial because we are taking into
account all possible toroidal compactifications. What is remarkable is
that, with this interpretation, Chern-Weil theory allows us to interpret
intersection products in terms of integrals of singular differential
forms (see theorems \ref{thm:3} and \ref{thm:5}). Moreover, there is a
Hilbert-Samuel type formula relating the asymptotic of the dimension
of the space of Jacobi forms with the self-intersection of the
b-divisor (Theorem \ref{thm:2}). 

The non-functoriality of the naive extension is exactly the height
jumping introduced by Hain (see  \cite{Hain:nfgmsc} and 
\cite{Pearlstein:dl2o}). 

\noindent
{\bfseries Acknowledgments} We have benefited from many discussions
with colleagues on the subject of this paper. We want to thank
S. Boucksom, R. de Jong, B. Edixhoven, D. Holmes, G. Freixas, A. von
Pippich, and 
M. Sombra for many useful 
discussions. We thank specially A. von Pippich for pointing to us the
Tornheim zeta function that is computed in \cite{Tornheim:hds} and
R. de Jong that
has computed independently the self-intersection product in Theorem
\ref{thm:1}, for sharing with us his work on the asymptotics of the
N\'eron height pairing \cite{HolmesJong:aNhp}, that gives a
complementary point of view on the results of this paper.  

This research has been conducted during visits of the authors to
the Humboldt University of Berlin, the ICMAT at Madrid and the
University of Barcelona. Our thanks go to these institutions for their
hospitality. 


\section{The universal elliptic surface}
\label{sec:ellsurface}

In the whole paper we fix an integer $N\ge 3$.
In this section we will revisit the definition of the universal elliptic
surface of level $N$ 
lying over the modular curve of level $N$. In particular, we will
recall the construction 
of its smooth toroidal compactification. For further details and
references the reader is referred \cite{Kramer:jacobi}.

\nnpar{The modular curve of level $N$.}
Let $\mathbb{H}$ denote the upper half-plane given by
\begin{align*}
\mathbb{H}:=\{\tau\in\mathbb{C}\,\vert\,\tau=\xi+i\eta,\,\eta>0\}
\end{align*}
and $\mathbb{H}^{*}:=\mathbb{H}\cup\mathbb{P}^{1}(\mathbb{Q})$ the extended
upper half-plane. The principal congruence subgroup
\begin{align*}
\Gamma(N):=\bigg\{\bigg(\begin{matrix}a&b\\c&d\end{matrix}\bigg)\in\mathrm{SL}_
{2}(\mathbb{Z})\,\bigg\vert\,a\equiv d\equiv 1\,\mathrm{mod}\,N,\,b\equiv c\equiv 0\,
\mathrm{mod}\,N\bigg\}
\end{align*}
of level $N$ acts in the usual way by fractional linear transformations on $\mathbb
{H}$; this action naturally extends to $\mathbb{H}^{*}$. The quotient space $X(N):=
\Gamma(N)\backslash\mathbb{H}^{*}$ is called the modular curve of level $N$;
it is the compactification of $Y(N):=\Gamma(N)\backslash\mathbb{H}$ by adding
the so-called cusps.

The modular curve $X(N)$ is a compact Riemann surface of genus 
\begin{align*}
g_{N}=1+\frac{N-6}{12}\frac{[\mathrm{SL}_{2}(\mathbb{Z}):\Gamma(N)]}{2N}\,,
\end{align*}
where the index of $\Gamma(N)$ in $\mathrm{SL}_{2}(\mathbb{Z})$ is given as
\begin{align*}
[\mathrm{SL}_{2}(\mathbb{Z}):\Gamma(N)]=N^{3}\prod\limits_{p\vert N}\bigg(1-
\frac{1}{p^{2}}\bigg).
\end{align*}
The number $p_{N}$ of cusps of $X(N)$ is given by
\begin{align*}
p_{N}=\frac{[\mathrm{SL}_{2}(\mathbb{Z}):\Gamma(N)]}{2N}\,;
\end{align*}
we denote the cusps by $P_{1}:=[\infty],P_{2},\ldots,P_{p_{N}}$. We recall that
$\Gamma(N)$ is a normal subgroup of $\mathrm{SL}_{2}(\mathbb{Z})$ and that
the quotient group $\mathrm{SL}_{2}(\mathbb{Z})/\Gamma(N)$ acts transitively
on the set of cusps (with stabilizers of order $N$). Therefore, it suffices in the
sequel to consider the cusp $P_{1}=[\infty]$. Since $N\ge 3$ the
group $\Gamma (N)$ is torsion-free. Therefore, $X(N)$ has no elliptic points.

We recall that the modular curve $X(N)$ is the moduli space of elliptic curves with
a full level $N$-structure. A point $[\tau]\in X(N)$ corresponds to the isomorphism
class of elliptic curves determined by $\mathbb{C}/(\mathbb{Z}\tau\oplus\mathbb
{Z})$ with $N$-torsion given by $(\mathbb{Z}\frac{\tau}{N}\oplus\mathbb{Z}\frac{1}
{N})/(\mathbb{Z}\tau\oplus\mathbb{Z})$.

\nnpar{The universal elliptic surface of level $N$.} We consider the product $\mathbb
{H}\times\mathbb{C}$ consisting of elements $(\tau,z)$ with $\tau\in\mathbb{H}$ and
$z=x+iy\in\mathbb{C}$. On $\mathbb{H}\times\mathbb{C}$ the semi-direct product
$\Gamma(N)\ltimes\mathbb{Z}^{2}$ acts by the assignment
\begin{align*}
\bigg[\bigg(\begin{matrix}a&b\\c&d\end{matrix}\bigg),(\lambda,\mu)\bigg](\tau,z):=\bigg
(\frac{a\tau+b}{c\tau+d},\frac{z+\lambda\tau+\mu}{c\tau+d}\bigg),
\end{align*}
where $\big(\begin{smallmatrix}a&b\\c&d\end{smallmatrix}\big)\in\Gamma(N)$ and
$(\lambda,\mu)\in\mathbb{Z}^{2}$. Since $N\ge 3$, the group $\Gamma
(N)$ is torsion-free. Hence, the action of $\Gamma(N)\ltimes
\mathbb{Z}^{2}$ on  $\mathbb{H}\times\mathbb{C}$ is free and the quotient space $E^{0}(N):=\Gamma(N)\ltimes
\mathbb{Z}^{2}\backslash\mathbb{H}\times\mathbb{C}$ is a smooth complex surface
together with a smooth surjective morphism
\begin{align*}
\pi^{0}\colon E^{0}(N)\longrightarrow Y(N)
\end{align*}
with fiber $(\pi^{0})^{-1}([\tau])=\mathbb{C}/(\mathbb{Z}\tau\oplus\mathbb{Z})$.

The surface $E^{0}(N)$ is known to extend to a compact complex surface $E(N)$
together with a surjective morphism
\begin{align*}
\pi\colon E(N)\longrightarrow X(N),
\end{align*}
the so-called universal elliptic surface of level $N$. To describe this extension, it
suffices to describe the fibers $\pi^{-1}(P_{j})$ above the cusps $P_{j}\in X(N)$
($j=1,\ldots,p_{N}$). These are given as $N$-gons, more precisely as
\begin{align*}
\pi^{-1}(P_{j})=\bigcup\limits_{\nu=0}^{N-1}\Theta_{j,\nu},
\end{align*}
where $\Theta_{j,\nu}\cong\mathbb{P}^{1}(\mathbb{C})$ is embedded into $E(N)$
with self-intersection number $-2$, while otherwise
\begin{align*}
\Theta_{j,\nu}\cdot\Theta_{j,\nu'}=\begin{cases}1\qquad\nu'=\nu\pm 1,\\0\qquad
\vert\nu-\nu'\vert\geq 2;\end{cases}
\end{align*}
here and subsequently, the indices have to be read modulo $N$.

In terms of local coordinates the situation above the cusp
$P_{1}=[\infty]$ can be
described as follows: The irreducible fiber $\Theta_{\nu}:=\Theta_{1,\nu}\subset
E(N)$ can be covered by two affine charts
$W_{\nu}^{0},W_{\nu}^{1}\subset E(N)$, where $W_{\nu}^{0}$ contains
the point $\Theta_{\nu}\cap \Theta_{\nu+1}$ and $W_{\nu}^{1}$ contains
the point $\Theta_{\nu}\cap \Theta_{\nu-1}$. Since $\Theta_{\nu}$ and
$\Theta_{\nu+1}$ intersect transversally, we can choose 
coordinates $u_{\nu},v_{\nu}$ on the chart $W_{\nu}^{0}$ in such a way
that $\Theta_{\nu}\vert_{W_{\nu}^{0}}$ is given by the equation
$v_{\nu}=0$ and $\Theta_{\nu+1}\vert_{W_{\nu}^{0}}$ by the equation $u_{\nu}=0$.
Using that $\Theta_{\nu}\cdot\Theta_{\nu}=-2$ we obtain that the
coordinates of $W_{\nu}^{1}$ are given 
by $u_{\nu}^{-1},u_{\nu}^{2}v_{\nu}$.
The open subset $W^{1}_{\nu+1}$ agrees with $W^{0}_{\nu}$. Hence we deduce
\begin{align*}
u_{\nu+1}=v_{\nu}^{-1},\quad v_{\nu+1}=u_{\nu}v_{\nu}^{2}.
\end{align*}
We finally note the relations
\begin{align}
\label{local}
u_{\nu}v_{\nu}=q_{N}:=\mathrm{e}^{2\pi i\tau/N},\quad u_{\nu}^{\nu+1}v_{\nu}^
{\nu}=\zeta:=\mathrm{e}^{2\pi iz}.
\end{align}
If we want to work with different cusps we will denote by
$W^{0}_{j,v}$ and $W^{1}_{j,v}$ the analogous affine charts around
points over the cusp $P_{j}$.  

We conclude by introducing the zero section
\begin{align*}
\varepsilon\colon X(N)\longrightarrow E(N)
\end{align*}
and by recalling that the arithmetic genus of $E(N)$ is given by
\begin{align*}
p_{\mathrm{a},N}=\frac{[\mathrm{SL}_{2}(\mathbb{Z}):\Gamma(N)]}{24}-1
=\frac{Np_{N}}{12}-1.
\end{align*}

\nnpar{Jacobi forms.} Modular forms can be interpreted 
as global sections of line bundles on the modular curve. The Jacobi
forms play a similar role for the universal elliptic curve. 

\begin{definition}
Let $k,m$ be
non-negative integers.
A holomorphic function $f\colon\mathbb{H}\times\mathbb{C}\rightarrow\mathbb
{C}$ is called \emph{Jacobi form of weight $k$, index $m$ for $\Gamma(N)$}, if
it satisfies the following properties:
\begin{itemize}
\item[(i)]
The function $f$ satisfies the functional equations
\begin{multline}\label{eq:1}
f\bigg(\frac{a\tau+b}{c\tau+d},\frac{z+\lambda\tau+\mu}{c\tau+d}\bigg)(c\tau+d)^
{-k}\,\times \\[2mm]
\times\exp\bigg(2\pi im\bigg(\lambda^{2}\tau+2\lambda z-\frac{c(z+\lambda\tau+
\mu)^{2}}{c\tau+d}\bigg)\bigg)=f(\tau,z)
\end{multline}
for all $\big[\big(\begin{smallmatrix}a&b\\c&d\end{smallmatrix}\big),(\lambda,\mu)
\big]\in\Gamma(N)\ltimes\mathbb{Z}^{2}$.
\item[(ii)]
At the cusp $P_{1}=[\infty]$, the function $f$ has a Fourier expansion of the form
\begin{align*}
f(\tau,z)=\sum\limits_{\substack{n\in\mathbb{N},\,r\in\mathbb{Z}\\4mn-Nr^{2}\geq
0}}c(n,r)q_{N}^{n}\zeta^{r},
\end{align*}
and similar Fourier expansions at the other cusps.
\end{itemize}
We denote the vector space of Jacobi forms of weight $k$, index $m$ for $\Gamma
(N)$ by $J_{k,m}(\Gamma(N))$.\\
If condition (ii) on the Fourier expansions is restricted to the summation over $n\in
\mathbb{N}_{>0}$ and $r\in\mathbb{Z}$ such that $4mn-Nr^{2}>0$, the function $f$ 
is called \emph{Jacobi cusp form of weight $k$, index $m$ for $\Gamma(N)$} and
the span of these functions is denoted by $J_{k,m}^{\cusp}(\Gamma(N))$.
\\
If condition (ii) on the Fourier expansions is dropped, the function $f$ is called
\emph{weak Jacobi form of weight $k$, index $m$ for $\Gamma(N)$}. The
span of these functions is denoted by $J_{k,m}^{\weak}(\Gamma(N))$. 
\end{definition}

\begin{remark}
  The condition \eqref{eq:1} is a cocycle condition that defines a
  line bundle $L_{k,m,N}$ on $E^{0}(N)$. 
The space of global sections of this line $H^{0}(E^{0}(N),L_{k,m,N})$ 
equals the space of weak Jacobi forms of weight $k$, index $m$ for $\Gamma
(N)$.
\end{remark}

\nnpar{Riemann theta functions.} The Riemann theta function
$\theta_{1,1}\colon \mathbb{H}\times\mathbb{C}\to \mathbb{C}$ is
defined by the convergent power series
\begin{align}
\label{theta}
\theta_{1,1}(\tau,z):=\sum\limits_{n\in\mathbb{Z}}\exp\bigg(\pi i\tau\bigg(n+\frac
{1}{2}\bigg)^{2}+2\pi i\bigg(z+\frac{1}{2}\bigg)\bigg(n+\frac{1}{2}\bigg)\bigg)
\end{align}
and satisfies the functional equation
\begin{align*}
&\theta_{1,1}\bigg(\frac{a\tau+b}{c\tau+d},\frac{z+\lambda\tau+\mu}{c\tau+d}
\bigg)(c\tau+d)^{-1/2}\,\times \\[2mm]
&\times\exp\bigg(\pi i\bigg(\lambda^{2}\tau+2\lambda z-\frac{c(z+\lambda\tau+
\mu)^{2}}{c\tau+d}\bigg)\bigg)=\chi\bigg(\begin{matrix}a&b\\c&d\end{matrix}\bigg)
\,\theta_{1,1}(\tau,z)
\end{align*}
for all $\big[\big(\begin{smallmatrix}a&b\\c&d\end{smallmatrix}\big),(\lambda,\mu)
\big]\in\mathrm{SL}_{2}(\mathbb{Z})\ltimes\mathbb{Z}^{2}$ with a character $\chi
(\cdot)$, which is an $8$-th root of unity.
Therefore, $\theta_{1,1}^{8}$ is a weak Jacobi form of weight $4$,
index $4$ for 
$\Gamma(1)=\mathrm{SL}_{2}(\mathbb{Z})$. Moreover, from the definition
power series \eqref{theta} it follows that $\theta_{1,1}^{8}$ is a
Jacobi form.

\nnpar{Dimension formulae.} We recall the dimension formulae for the
space of Jacobi forms. For simplicity, we restrict ourselves to the case
$m=k=4\ell$.
We denote by $j\colon E^{0}(N)\to E(N)$ the open immersion.
From \cite{Kramer:jacobi}, we cite the following result.

\begin{proposition}
There is a distinguished subsheaf $\mathcal{F}_{\ell}$ of the sheaf
$j_{\ast}L_{4\ell,4\ell,N}$ such that we have an isomorphism 
\begin{align*}
J_{4\ell,4\ell}^{\cusp}\big(\Gamma(N)\big)\cong H^{0}\big(E(N),\mathcal{F}_
{\ell}\big).
\end{align*}
In particular, the dimension of $J_{4\ell,4\ell}^{\cusp}(\Gamma(N))$
is given, when $N$ divides $4\ell$, by
\begin{align*}
\dim J_{4\ell,4\ell}^{\cusp}\big(&\Gamma(N)\big)\\
&=p_N\bigg(
\frac{8N\ell^2}{3}-N\ell-
\frac{N}{4}Q\big(\frac{16\ell}{N}\big)-
\frac{N}{2}\hspace*{-2mm}\sum\limits_{\substack{\Delta\mid
    16\ell/N,\,\Delta<0\\16\ell/(N\Delta)\,\mathrm{squarefree}}}
H(\Delta)\bigg) \\
&=\frac{8Np_N}{3}\ell^{2}+o\big(\ell^{2}\big),
\end{align*}
where $Q(n)$ denotes the largest integer whose square divides $n$ and $H
(\Delta)$ is the Hurwitz class number. 
\end{proposition}
\begin{proof}
The first statement is \cite[Theorem~2.6]{Kramer:jacobi}, the second
statement is \cite[Theorem~3.8]{Kramer:jacobi}, noting that
\begin{align*}
[\mathrm{SL}_{2}(\mathbb{Z}):\Gamma(N)]=2Np_{N}.
\end{align*}
To prove the assymptotic
estimate one uses that $Q(n)$ is at most $\sqrt{n}$, that the number
of divisors of an integer $n$ is $o(n^{\varepsilon })$ for any
$\varepsilon >0$ and that, by the Brauer-Siegel theorem, the Hurwitz
class number $H(\Delta )$ is
$o(|\Delta| ^{1/2+\varepsilon })$ for any $\varepsilon >0$. 
\end{proof}

\begin{remark} \label{rem:1}
Since $\dim
J_{4\ell,4\ell}\big(\Gamma(N)\big)-J^{\cusp}_{4\ell,4\ell}\big(\Gamma(N)\big)$
grows at most linearly with $\ell$, 
 we also have the asymptotic formula
\begin{align*}
\dim J_{4\ell,4\ell}\big(\Gamma(N)\big)=\frac{8Np_N}{3}\ell^{2}+o\big(\ell^{2}\big).
\end{align*}
\end{remark}

\nnpar{Translation invariant metric.} Here we recall the translation
invariant metric
on the line bundle $L_{k,m,N}$.

\begin{definition} For $f\in J_{k,m}^{\weak}(\Gamma(N))$, we define
\begin{align*}
\Vert f(\tau,z)\Vert^{2}:=\vert f(\tau,z)\vert^{2}\exp(-4\pi my^{2}/\eta)\eta^{k},
\end{align*}
where we recall that $\eta=\mathrm{Im}(\tau)$ and $y=\mathrm{Im}(z)$. 
\end{definition}

\begin{lemma}
For $f\in J_{k,m}^{\weak}(\Gamma(N))$, we have
\begin{align*}
\Bigg\Vert
f\bigg(\frac{a\tau+b}{c\tau+d},\frac{z+\lambda\tau+\mu}{c\tau+d}\bigg) 
\Bigg\Vert^{2}=\Vert f(\tau,z)\Vert^{2}
\end{align*}
for all
$\big[\big(\begin{smallmatrix}a&b\\c&d\end{smallmatrix}\big),(\lambda,\mu) 
\big]\in\Gamma(N)\ltimes\mathbb{Z}^{2}$. In particular, this shows
that $\Vert\cdot 
\Vert$ induces a hermitian metric on the line bundle $L_{k,m,N}$.
\end{lemma}
\begin{proof}
This is a straightforward calculation.
\end{proof}

\begin{lemma}\label{metriclocal}
Locally, in the affine chart $W_{\nu}^{0}$ over the cusp $P_{1}=[\infty]$, the
hermitian metric $\Vert\cdot\Vert$ is described by the formula
\begin{align*}
&\log\big(\Vert
f(\tau,z)\Vert^{2}\big)\Big\vert_{W_{\nu}^0}=\log\big(\vert f(\tau,z)
\vert^{2}\big)\Big\vert_{W_{\nu}^0} \\
&\hspace*{5mm}+\frac{m}{N}\bigg((\nu+1)^{2}\log(u_{\nu}\bar{u}_{\nu})+\nu^
{2}\log(v_{\nu}\bar{v}_{\nu})-\frac{\log(u_{\nu}\bar{u}_{\nu})\log(v_{\nu}\bar{v}_
{\nu})}{\log(u_{\nu}\bar{u}_{\nu})+\log(v_{\nu}\bar{v}_{\nu})}\bigg) \\
&\hspace*{5mm}+k\log\bigg(-\frac{N}{4\pi}\big(\log(u_{\nu}\bar{u}_{\nu})+\log
(v_{\nu}\bar{v}_{\nu})\big)\bigg).
\end{align*}
\end{lemma}
\begin{proof}
Taking absolute values, we derive from \eqref{local}
\begin{align*}
&\eta=-\frac{N}{2\pi}\log\vert q_{N}\vert=-\frac{N}{2\pi}\log\vert u_{\nu}v_{\nu}
\vert \\[1mm]
&\hspace*{2mm}=-\frac{N}{4\pi}\big(\log(u_{\nu}\bar{u}_{\nu})+\log(v_{\nu}
\bar{v}_{\nu})\big), \\[2mm]
&y=-\frac{1}{2\pi}\log\vert\zeta\vert=-\frac{1}{2\pi}\log\vert u_{\nu}^{\nu+1}v_
{\nu}^{\nu}\vert \\[1mm]
&\hspace*{2mm}=-\frac{1}{4\pi}\big((\nu+1)\log(u_{\nu}\bar{u}_{\nu})+\nu
\log(v_{\nu}\bar{v}_{\nu})\big).
\end{align*}
With these formulae we compute
\begin{align*}
&-\frac{4\pi my^{2}}{\eta}=\frac{m}{N}\frac{\big((\nu+1)\log(u_{\nu}\bar{u}_
{\nu})+\nu\log(v_{\nu}\bar{v}_{\nu})\big)^{2}}{\log(u_{\nu}\bar{u}_{\nu})+
\log(v_{\nu}\bar{v}_{\nu})}= \\[1mm]
&\frac{m}{N}\bigg((\nu+1)^{2}\log(u_{\nu}\bar{u}_{\nu})+\nu^{2}\log(v_{\nu}
\bar{v}_{\nu})-\frac{\log(u_{\nu}\bar{u}_{\nu})\log(v_{\nu}\bar{v}_{\nu})}{\log
(u_{\nu}\bar{u}_{\nu})+\log(v_{\nu}\bar{v}_{\nu})}\bigg).
\end{align*}
From this the proof follows immediately from the definition of the hermitian
metric $\Vert\cdot\Vert$.
\end{proof}


\section{Mumford-Lear extensions and b-divisors}
\label{sec:mumf-lear-comp}

In this section we will introduce Mumford-Lear extensions of a line
bundle and relate them with b-divisors. We first recall the different
notions of growth for metrics and differential forms that will be
useful in the sequel. 

\nnpar{Notations.}
Let $X$ be a complex algebraic manifold of dimension $d$ and
$D$ a normal crossing divisor of $X$. Write $U=X\setminus D$,
and let $j\colon U\longrightarrow X$ be the inclusion.

Let $V$ be an open coordinate subset of $X$ with coordinates
$z_{1},\dots,z_{d}$; we put $r_{i}=|z_{i}|$. We say that $V$
\emph{is adapted to $D$}, if the divisor $D$ is locally given
by the equation $z_{1}\cdots z_{k}=0$. We assume that the 
coordinate neighborhood $V$ is small enough; more precisely, 
we will assume that all the coordinates satisfy $r_{i}<1/e^{e}$, 
which implies that $\log 1/r_{i}>e$ and $\log(\log 1/r_{i})>1$.

If $f$ and $g$ are two functions with non-negative real values,
we will write $f\prec g$, if there exists a constant $C>0$
such that $f(x)\le C\cdot g(x)$ for all $x$ in the domain of
definition under consideration.

\nnpar{log-log growth forms.}
\begin{definition}
\label{def:loglog}
We say that a smooth complex function $f$ on $X\setminus D$ 
has \emph{log-log growth along $D$}, if we have
\begin{equation}
\label{eq:loglog1}
|f(z_{1},\dots,z_{d})|\prec\prod_{i=1}^{k}\log(\log(1/r_{i}))^
{M}
\end{equation}
for any coordinate subset $V$ adapted to $D$ and some positive
integer $M$. The \emph{sheaf of differential forms on $X$ with
log-log growth along $D$} is the subalgebra of $j_{\ast}\mathscr
{E}^{\ast}_{U}$ generated, in each coordinate neighborhood $V$
adapted to $D$, by the functions with log-log growth along $D$
and the differentials
\begin{alignat*}{2}
&\frac{\dd z_{i}}{z_{i}\log(1/r_{i})},\,\frac{\dd\bar{z}_{i}}
{\bar{z}_{i}\log(1/r_{i})},&\qquad\text{for }i&=1,\dots,k, \\
&\dd z_{i},\,\dd\bar{z}_{i},&\qquad\text{for }i&=k+1,\dots,d.
\end{alignat*}
If $D$ is clear form the context, 
a differential form with log-log growth along $D$ will be called
a \emph{log-log growth form}.
\end{definition}

\nnpar{Dolbeault algebra of pre-log-log forms.}
Clearly, the forms with log-log growth form an algebra but not a
differential algebra. To remedy this we impose conditions on the
derivatives as well.  

\begin{definition}
A log-log growth form $\omega$ such that $\partial\omega$,
$\bar{\partial}\omega$ and $\partial\bar{\partial}\omega$
are also log-log growth forms is called a \emph{pre-log-log
form (along $D$)}. The sheaf of pre-log-log forms is the subalgebra of
$j_{\ast}\mathscr{E}^{\ast}_{U}$ generated by the pre-log-log
forms. We will denote this complex by $\mathscr{E}^{\ast}_{X}
\langle\langle D\rangle\rangle_{\w}$. The pre-log-log forms of degree
zero are called \emph{pre-log-log} functions.
\end{definition}

The sheaf $\mathscr{E}^{\ast}_{X}\langle\langle D\rangle\rangle_
{\w}$, together with its real structure, its bigrading, and the
usual differential operators $\partial$, $\bar{\partial}$ is 
easily shown to be a sheaf of Dolbeault algebras. Moreover, 
it is the maximal subsheaf of Dolbeault algebras of the sheaf 
of differential forms with log-log growth.

\nnpar{Metrics with logarithmic growth and pre-log metrics.}
Let $L$ be a line bundle on $X$ and let $\|\cdot\|$ be a smooth hermitian
metric on $L| _{U}$.

\begin{definition} \label{def:1} We will say that the metric $\|\cdot\|$ has
  \emph{logarithmic growth (along $D$)} if, for every point $x\in X$,
  there is a coordinate neighbourhood $V$ of $x$ adapted to $D$, a nowhere zero
  regular section $s$ of $L$ on $V$, and an integer $M\ge 0$ such that
  \begin{equation}
    \prod_{i=1}^{k}\log(1/r_{i})^
    {-M}\prec
    \|s(z_{1},\dots,z_{d})\|\prec\prod_{i=1}^{k}\log(1/r_{i})^
    {M}
  \end{equation}
\end{definition}

\begin{definition} \label{def:2} We will say that the metric
  $\|\cdot\|$ is a 
  \emph{pre-log metric (along $D$)} if it has logarithmic growth and,
  for every rational section $s$ of $L$, the function $\log\|s\|$ is a
  pre-log-log form along $D\setminus \dv(s)$ on $X\setminus \dv(s)$.
\end{definition}

\nnpar{Mumford-Lear extensions.} We are now able to define
Mumford-Lear extensions. For the remainder of the section we fix a
complex algebraic manifold $X$ of dimension $d$, 
$D$ and $U$ as before, and we also fix a hermitian line bundle $\overline
L=(L,\|\cdot\|)$ on $U$. 

\begin{definition} \label{def:3}
  We say that \emph{$\overline L$ admits
    a Mumford-Lear extension to $X$} if there is an integer $e\ge 1$, a
  line bundle $\mathcal {L}$ on $X$, an algebraic subset $S\subset X$
  of codimension at least 2 that is contained in $D$, a smooth
  hermitian metric $\|\cdot\|$ on $\mathcal {L}|_{U}$ that has
  logarithmic growth along $D\setminus S$ and an
  isometry $\alpha \colon
  (L,\|\cdot\|)^{\otimes e}\to (\mathcal {L}|_U,\|\cdot\|)$. The $5$-tuple
  $(e,\mathcal{L},S,\|\cdot\|,\alpha )$ is called a \emph{Mumford-Lear
    extension} of $\overline L$. When the isomorphism $\alpha $, the
  metric and the set $S$ can be deduced by the context, we will denote
  the Mumford-Lear extension by $(e,\mathcal{L})$. If $e=1$, we will
  denote it by the line bundle $\mathcal{L}$. 
\end{definition}

\begin{remark}
  The name Mumford-Lear extension arises because they generalize (in
  the case of line bundles) the extensions of hermitian vector bundles
  considered by Mumford in 
  \cite{Mumford:Hptncc} and the extensions of line bundles considered
  by Lear in his thesis \cite{Lear:enfahp}.
\end{remark}

The Mumford-Lear extensions satisfy the following unicity property.

\begin{proposition}\label{prop:1} Assume that $\overline L$ admits
    a Mumford-Lear extension to $X$.
  Let $(e_{1},\mathcal{L}_{1},S_{1},\|\cdot\|_{1},\alpha
  _{1})$ and $(e_{2},\mathcal{L}_{2},S_{2},\|\cdot\|_{2},\alpha
  _{2})$ be two Mumford-Lear extensions of $\overline L$ to $X$. Then
  there is a unique isomorphism $$\varphi\colon
  \mathcal{L}_{1}^{\otimes e_{2}}\to
  \mathcal{L}_{2}^{\otimes e_{1}}$$ such that the diagram
  \begin{displaymath}
    \xymatrix{& \mathcal{L}_{1}^{\otimes e_{2}}|_{U}
      \ar[dd]^{\varphi| _{U}}\\ 
      L^{\otimes e_{1}e_{2}}\ar[ur]^{\alpha _{1}^{\otimes e_{2}}}
      \ar[dr]^{\alpha _{2}^{\otimes e_{1}}}&\\
      & \mathcal{L}_{2}^{\otimes e_{1}}|_{U}
    }
  \end{displaymath}
  is commutative.
\end{proposition}
\begin{proof}
  The composition $\alpha_{2} ^{\otimes e_{1}}\circ (\alpha
  _{1}^{-1})^{\otimes e_{2}}$ defines an isomophism between the line
  bundles 
  $\mathcal{L}_{1}^{\otimes e_{2}}|_{U}$ and
  $\mathcal{L}_{2}^{\otimes e_{1}}|_{U}$ that is the only one that
  makes the diagram in the theorem commutative. Put $S=S_{1}\cup
  S_{2}$. The proof of
  \cite[Proposition 1.3]{Mumford:Hptncc} shows that this isomorphism
  extends uniquely to an isomorphism $\varphi_{1}\colon
  \mathcal{L}_{1}^{\otimes e_{2}}|_{X\setminus S}\to
  \mathcal{L}_{2}^{\otimes e_{1}}|_{X\setminus S}$. Since $X$ is
  smooth and $S$ has codimension 2, the isomorphism $\varphi_{1}$
  extends to a unique 
  isomorphism $\varphi\colon
  \mathcal{L}_{1}^{\otimes e_{2}}\to
  \mathcal{L}_{2}^{\otimes e_{1}}$ satisfying the condition of the
  proposition. 
\end{proof}

The next result is an immediate consequence of Proposition \ref{prop:1}. 

\begin{corollary}
  Assume the hypothesis of the previous proposition. Let $s$ be a
  rational section of $L$, so $\alpha
  _{1}(s^{\otimes  e_{1}})^{\otimes e_{2}}$ and $\alpha
  _{2}(s^{\otimes  e_{2}})^{\otimes e_{1}}$ are rational sections of
  $\mathcal{L}_{1}^{\otimes e_{2}}$ and $\mathcal{L}_{2}^{\otimes
    e_{1}}$, respectively. Then
  \begin{displaymath}
    \dv(\alpha
    _{1}(s^{\otimes  e_{1}})^{\otimes e_{2}})=
    \dv(\alpha
  _{2}(s^{\otimes  e_{2}})^{\otimes e_{1}})
  \end{displaymath}
  as Cartier divisors on $X$.
\end{corollary}
\begin{proof}  
\end{proof}

This result allows us to associate to each rational section of $L$ a
$\QQ$-Cartier divisor on $X$. We will denote by $\QCa(X)$ the group of
$\QQ$-Cartier divisors of $X$.

\begin{definition} \label{def:4}
  Assume that $\overline L$ admits
    a Mumford-Lear extension to $X$ and let
    $(e,\mathcal{L},S,\|\cdot\|,\alpha )$ be one such extension.  Let
    $s$ be a rational section of $L$. Then we define \emph{the 
    divisor of $s$ on $X$} as the $\QQ$-Cartier divisor
  \begin{displaymath}
    \dv_{X}(s)=\frac{1}{e}\dv(\alpha (s^{\otimes e}))\in \QCa(X),
  \end{displaymath}
  where $\dv(\alpha (s^{\otimes e}))$ is the divisor of $\alpha
  (s^{\otimes e})$ viewed as a rational section of $\mathcal{L}$ on
  the whole $X$.
\end{definition}

\nnpar{Mumford-Lear extensions and birational maps.} We now consider
Mumford-Lear extensions on different birational models of $X$.

\begin{notation} \label{def:5}
  Let
  $\mathcal{C}$ be the category whose objects 
  are pairs $(Y,\pi _{Y})$, where $Y$ is a smooth complex
  variety and $\pi _{Y}\colon Y\to X$ is a proper birational map, and
  whose morphisms are maps $\varphi\colon Y\to Z$ such that $\pi
  _{Z}\circ \varphi =\pi _{Y}$.
  We denote by $\Bir(X)$ the set of isomorphism classes in
  $\mathcal{C}$. Since the set of morphisms between two objects of
  $\mathcal{C}$ is either empty or contains a single element, the set
  $\Bir(X)$ is itself a small category equivalent
  to $\mathcal{C}$. In fact $\Bir(X)$ is a directed set.
  As a shorthand, an element $(Y,\pi _{Y})$
  of $\Bir(X)$ 
  will be denoted by the variety $Y$, the map $\pi _{Y}$ being
  implicit. For an object $Y$ of $\Bir(X)$ we will denote $U_{Y}=\pi
  _{Y}^{-1}(U)$ and $D_{Y}=\pi
  _{Y}^{-1}(D)$. By abuse of notation we will denote also by $\pi _{Y}$ the
  induced proper birational map from $U_{Y}$ to $U$. Finally, we will
  denote by $\Bir'(X)$ the subset consisting of the elements $Y$ with
  $D_{Y}$ a normal crossing divisor. This is a cofinal subset.
\end{notation}

\begin{definition} \label{def:6} We say that $\overline L$ \emph{admits all
  Mumford-Lear extensions over $X$} if, for every object $Y$ of
  $\Bir'(X)$, the hermitian line bundle $\pi _{Y}^{\ast}\overline L$
  on $U_{Y}$ admits a Mumford-Lear extension to $Y$.
\end{definition}

\begin{definition}\label{def:7}
  Assume that $\overline L$ admits all
  Mumford-Lear extensions over $X$. For every $Y\in \Bir'(X)$, let
  $(e',\mathcal{L}',S',\|\cdot\|',\alpha' )$ be 
  a
  Mumford-Lear extension of $\pi _{Y}^{\ast}\overline L$ 
  to $Y$. Then the \emph{divisor of $s$ on $Y$} is
  defined as
  \begin{displaymath}
    \dv_{Y}(s)=\frac{1}{e'}\dv(\alpha' (s^{\otimes e'}))\in \QCa(Y).
  \end{displaymath}    
\end{definition}

The $\QQ$-Cartier divisors of Definition \ref{def:7} do not need to
be compatible with inverse images. As we will see in concrete
examples, it may happen that there are maps $\varphi\colon Y\to Z$ in
$\Bir'(X)$ such that
\begin{displaymath}
  \varphi ^{\ast}\dv_{Z}(s)\not = \dv_{Y}(s).
\end{displaymath}
This lack of compatibility with inverse images is related with the
phenomenon of height jumping (see \cite{Hain:nfgmsc} and
\cite{Pearlstein:dl2o} for a discussion of height jumping).  

In contrast, the divisors associated with Mumford-Lear extensions
are compatible with direct images.

\begin{proposition}
  Assume that $\overline L$ admits all Mumford-Lear extensions over
  $X$. 
  Let $\varphi\colon Y\to Z$ be a map in $\Bir'(X)$ and $s$ a section
  of $L$. Then
  \begin{displaymath}
    \varphi_{\ast}\dv_{Y}(s)=\dv_{Z}(s).
  \end{displaymath}
\end{proposition}
\begin{proof}
  Let $T$ be the subset of $Z$ where $\varphi^{-1}$ is not defined. Since $Z$
  is smooth, hence normal, by Zariski's main theorem, $T$ has
  codimension at least 2. Write
  $W=Z\setminus T$ and let $U'=U_{Z}\cap W$. 
  Then $\overline L$
  induces a line bundle on $U'$ that admits a Mumford-Lear extension
  to $W$.  
  
  Since $T$ has codimension 2, the restriction map  $$\QCa(Z)\to
  \QCa(W)$$ is an isomorphism. Moreover, using the definition is easy
  to see that
  \begin{displaymath}
    \dv_{Y}(s)| _{W}=\dv_{W}(s)=\dv_{Z}(s)| _{W}.
  \end{displaymath}
  Thus the proposition follows from the commutativity of the diagram
  \begin{displaymath}
    \xymatrix{ \QCa(Y) \ar[r] \ar[d]^{\varphi_{\ast}}& \QCa(W)
      \ar@{=}[d]\\
      \QCa(Z) \ar[r]^{\simeq} & \QCa(W)
    }.
  \end{displaymath}
\end{proof}

\nnpar{B-Divisors.} Recall that the Zariski-Riemann space of $X$ is
the projective limit
\begin{displaymath}
  \mathfrak{X}=\lim_{\substack{\longleftarrow\\\Bir(X)}}Y.
\end{displaymath}
We are not going to use the structure of this space,
which is introduced here merely in order to make later definitions more
suggestive. 

For the definition of b-divisors, we will follow the point of view of
\cite{BoucksomFavreJonsson:dvpt}. The groups $\QCa(Y)$, $Y\in \Bir(X)$,
form a projective system with the push-forward morphisms and an
inductive system with the pull-back morphisms. We define the group of
$\QQ$-Cartier divisor on $\mathfrak{X}$ as the inductive limit 
\begin{displaymath}
  \QCa(\mathfrak{X})=\lim_{\substack{\longrightarrow\\\Bir(X)^{\op}}} \QCa(Y)
\end{displaymath}
and the group of $\QQ$-Weil divisors on $\mathfrak{X}$ as the
projective limit  
\begin{displaymath}
  \QWe(\mathfrak{X})=\lim_{\substack{\longleftarrow\\\Bir(X)}} \QCa(Y).
\end{displaymath}
Since, for any map $\varphi$ in $\Bir(X)$ the composition
$\varphi_{\ast}\circ \varphi^{\ast}$ is the identity, it is easy to
see that there is a map $\QCa(X)\to \QWe(X)$.  Note also that, since
$\Bir'(X)$ is cofinal, the above projective and inductive limit can be
taken over $\Bir'(X)$.

The group of $\QQ$-Weil divisors of $\mathfrak{X}$ is closely
related to the group of b-divisors of $X$ defined in
\cite{Shokurov:pf}. Thus a $\QQ$-Weil divisors of $\mathfrak{X}$ will
be called a b-divisors of $X$.

The following definition makes sense thanks to Proposition
\ref{prop:1}.

\begin{definition}
  Assume that $\overline L$ admits all Mumford-Lear extensions over
  $X$. Let $s$ be rational section of $L$. Then the \emph{b-divisor
    associated} to $s$ is
  \begin{displaymath}
    \bdv(s)=(\dv_{Y}(s))_{Y\in \Bir'(X)}\in \QWe(\mathfrak{X}).
  \end{displaymath}
  When it is needed to specify with respect to which metric we are
  compactifying the divisor, we will write $\bdv(s,\|\cdot\|)$.
\end{definition}

\nnpar{Integrable b-divisors.} From now on we restrict ourselves to the
case when $X$ is a surface. We want to extend the intersection product
of divisors as much as possible to b-divisors. 

It is clear that there is an intersection pairing
\begin{displaymath}
  \QCa(\mathfrak{X})\times \QWe(\mathfrak{X})\to \QQ
\end{displaymath}
defined as follows. Let $C\in \QCa(\mathfrak{X})$ and $E\in
\QWe(\mathfrak{X})$. Then there is an object $Y\in \Bir'(X)$ and a
divisor
$C_{Y}\in \QCa(X)$ such that $C$ is the image of $C_{Y}$. Let $E_{Y}$
be the component of $E$ on $Y$. Then, by
the projection formula, the
intersection product $C_{Y}\cdot E_{Y}$ does not depend on the choice of
$Y$. Thus we define
\begin{displaymath}
  C\cdot E=C_{Y}\cdot E_{Y}.
\end{displaymath}

But, in general, we can not define the intersection product of two
elements of $\QWe(\mathfrak{X})$. The following definition is the
analogue for b-divisors of the concept of $L^{2}$-function. Recall
that, since $\Bir'(X)$ is a directed set, it is in particular a net.

\begin{definition}
  A divisor $C=(C_{Y})_{Y\in \Bir'(X)}\in \QWe(\mathfrak{X})$ is
  called integrable if the limit
  \begin{displaymath}
    \lim_{\substack{\longrightarrow\\\Bir'(X)}}C_{Y}\cdot C_{Y}
  \end{displaymath}
  exists and is finite.
\end{definition}

\begin{proposition}
  Let $C_{1}, C_{2}\in \QWe(\mathfrak{X})$. If $C_{1}$ and $C_{2}$ are
  integrable, then 
  \begin{displaymath}
    \lim_{\substack{\longrightarrow\\\Bir'(X)}}C_{1,Y}\cdot C_{2,Y}
  \end{displaymath}
  exists and is finite.
\end{proposition}
\begin{proof}
  Let $C=(C_{Y})\in \QWe(\mathfrak{X})$ and let $\varphi\colon Y\to Z$
  be an arrow in $\Bir'(X)$. Since 
  $\varphi_{\ast}C_{Y}=C_{Z}$, we deduce that
  \begin{displaymath}
    C_{Y}=\varphi^{\ast} C_{Z}+E,
  \end{displaymath}
  where $E$ is an exceptional divisor for the map
  $\varphi$. Hence,
  \begin{displaymath}
    C_{Y}\cdot C_{Y}=(\varphi^{\ast}
    C_{Z}+E)\cdot(\varphi^{\ast} C_{Z}+E)=C_{Z}\cdot
    C_{Z}+E\cdot E.
  \end{displaymath}
  Thus, by the Hodge index theorem,
  \begin{displaymath}
    C_{Y}\cdot C_{Y}-C_{Z}\cdot
    C_{Z}=E\cdot E\le 0.
  \end{displaymath}
  Hence
  \begin{multline*}
    0\ge (C_{1,Y}\pm C_{2,Y})^2-(C_{1,Z}\pm C_{2,Z})^{2}\\=
    C_{1,Y}^{2}-C_{1,Z}^{2}+C_{2,Y}^{2}-C_{2,Z}^{2}\pm 2(C_{1,Y}\cdot
    C_{2,Y}-C_{1,Z}\cdot
    C_{2,Z}).
  \end{multline*}
  Therefore
  \begin{multline*}
    C_{1,Y}\cdot
    C_{2,Y}-C_{1,Z}\cdot
    C_{2,Z}\le
    -\frac{1}{2}(C_{1,Y}^{2}-C_{1,Z}^{2}+C_{2,Y}^{2}-C_{2,Z}^{2})\\
    =\frac{1}{2}(|C_{1,Y}^{2}-C_{1,Z}^{2}|+|C_{2,Y}^{2}-C_{2,Z}^{2}|)
  \end{multline*}
  and 
  \begin{multline*}
    C_{1,Y}\cdot
    C_{2,Y}-C_{1,Z}\cdot
    C_{2,Z}\ge
    \frac{1}{2}(C_{1,Y}^{2}-C_{1,Z}^{2}+C_{2,Y}^{2}-C_{2,Z}^{2})\\
    =-\frac{1}{2}(|C_{1,Y}^{2}-C_{1,Z}^{2}|+|C_{2,Y}^{2}-C_{2,Z}^{2}|).
  \end{multline*}
  Thus
  \begin{displaymath}
        |C_{1,Y}\cdot
    C_{2,Y}-C_{1,Z}\cdot
    C_{2,Z}|\le
    \frac{1}{2}(|C_{1,Y}^{2}-C_{1,Z}^{2}|+|C_{2,Y}^{2}-C_{2,Z}^{2}|)
  \end{displaymath}
  Thus the convergence of $(C_{1,Y}^{2})_{Y}$ and $(C_{2,Y}^{2})_{Y}$
  implies the convergence of $(C_{1,Y}\cdot
    C_{2,Y})_{Y}$.
\end{proof}

\section{The Mumford-Lear extension of the line bundle of Jacobi
  forms}
\label{sec:mumf-lear-extens-jacobi}

In this section we will study the Mumford-Lear extensions of the line
bundle of Jacobi forms.

\nnpar{The functions $f_{n,m}$.} We first study a family of
functions that will be useful latter. 
Let $(n,m)$ be a pair of coprime
positive integers. Let $u,v$ be coordinates of $\CC^{2}$ and denote
$U_{n,m}\subset \mathcal{\CC}^{2}$ the 
open subset defined by $|uv|<1$. Let $D\subset U_{n,m}$ be the normal
crossing divisor of equation $uv=0$.

\begin{proposition}\label{prop:3}
Let $f_{n,m}$ be the function on $U_{n,m}$ given by
\begin{displaymath}
  f_{n,m}(u,v)=\frac{1}{nm}\frac{\log(u\overline u)\log(v\overline v)}
  {n\log(u\overline u)+m\log(v\overline v)}
\end{displaymath}
This function satisfies the following properties.
\begin{enumerate}[(i)]
\item \label{item:1} The function $f_{n,m}$ is a pre-log-log function along
  $D\setminus \{(0,0)\}$.
\item \label{item:2} The equality  
  $\partial\overline \partial f_{n,m}\land \partial\overline \partial
  f_{n,m}=0$ holds. The differential forms $f_{n,m}$, $\partial
  f_{n,m}$, $\overline \partial
  f_{n,m}$, and $\partial\overline \partial f_{n,m}$ and all the
  products between them are locally integrable.  Moreover, any product
  of $\partial\overline \partial f_{n,m}$ with a 
  pre-log-log form along $D$ is also locally integrable.
\item \label{item:3} Let $\pi \colon U_{n,n+m}\to U_{n,m}$ be the map given by
  $(s,t)\mapsto (st,t)$. Note that $U_{n,n+m}$ is a chart of the
  blow-up of $U_{n,m}$ along $(0,0)$. Then
  \begin{displaymath}
    \pi ^{\ast}f_{n,m}(s,t)=\frac{1}{nm(n+m)}\log (t\overline
    t)+f_{n,n+m}(s,t). 
  \end{displaymath}
\end{enumerate}
\end{proposition}
\begin{proof}
  Put
  \begin{gather*}
    P_{n,m}(u,v)=n\log (u\overline u) + m\log (v\overline v) \\
    a=\log (v\overline v) \frac{\dd u}{u},\qquad b=\log (u\overline u)
    \frac{\dd v}{v}.
  \end{gather*}
With these notations, we have
\begin{align}
  \partial f_{n,m}&=\frac{1}{nm P_{n,m}^{2}}(n\log(u\overline
  u)b+m\log(v\overline v)a),\label{eq:4}\\
  \partial\overline \partial f_{n,m}&=\frac{2}{P_{n,m}^{3}}
  (b-a)\land (\overline a-\overline b),\label{eq:5}\\
  \partial f_{n,m}\land \partial\overline \partial f_{n,m}&=
  \frac{2}{nmP_{n,m}^{4}}a\land b\land(\overline a-\overline b).\label{eq:10}
\end{align}
From equation \eqref{eq:5}, it follows that $
  \partial\overline \partial f_{n,m}\land \partial\overline \partial
  f_{n,m}=0.$

We now prove \eqref{item:1}.
Let $p=(0,v_{0})\in D\setminus \{(0,0)\}$. Let $U$ be a neighborhood
of $p$ such that  $|\log(v\overline v)|\le K$, $u\overline u< 1$ and
\begin{displaymath}
  n|\log(u\overline u)|\ge 2mK\ge 2m |\log(v\overline v)|,
\end{displaymath}
for some positive constant $K$. Therefore, on all the points of $U$,
the estimate
\begin{displaymath}
  |P_{n,m}|\ge \frac{n}{2}|\log (u\overline u)|
\end{displaymath}
holds. Then, for $(u,v)\in U$,
\begin{equation}\label{eq:2}
  |f_{n,m}(u,v)|\le \frac{2}{n^{2}m}\frac{| \log (u\overline u)\log
    (v\overline v)|}{| \log (u\overline u)|}\le \frac{2K}{n^{2}m}.
\end{equation}
Similarly, if $t_{1}$ and $t_{2}$ are smooth tangent vectors on $U$
with bounded coefficients, from \eqref{eq:4} and \eqref{eq:5}, we derive
\begin{align}
  \label{eq:3}
  |\partial f_{n,m}(t_{1})|&\le \frac{C_{1}}{|\log(u\overline
    u)|^{2}|u|}\\
  \label{eq:6}
  |\partial \overline \partial f_{n,m}(t_{1},t_{2})|&\le 
  \frac{C_{2}}{|\log(u\overline
    u)|^{3}|u|^{2}}
\end{align}
for suitable positive constants $C_{1}$ and $C_{2}$. The estimates
\eqref{eq:2}, \eqref{eq:3} and \eqref{eq:6} show that $f_{n,m}$ is a
pre-log-log function. Thus we have proved \eqref{item:1}.

Since pre-log-log forms are always locally integrable
(cf. \cite[Proposition 7.6]{BurgosKramerKuehn:cacg}), in order to check
\eqref{item:2}, it is only necessary to study a neighborhood of the
point $(0,0)$. Thus we restrict ourselves to the open $V$ defined by
$|u|<1/e$ and $|v|<1/e$.

We show the local integrability of a form of the type
$\partial\overline \partial f_{n,m}\land \varphi$
for a pre-log-log form $\varphi$, being the other cases analogous. 

By the definition of pre-log-log forms, equation \eqref{eq:5} shows
that $\partial\overline \partial f_{n,m}\land \varphi=g(u,v)\dd u\land \dd
\overline u\land \dd v\land \dd
\overline v$, with  $g$ a function satisfying
\begin{displaymath}
  |g(u,v)|\le C_{1}\frac{ |{\log (\log (u\overline u))\log (\log
    (v\overline v))}|^{M}}{|P_{n,m}^{3}u\overline u v\overline v|}.  
\end{displaymath}
for certain positive constants $C_{1}$ and $M$. Using the geometric
vs. arithmetic 
mean inequality, and the fact that the logarithm grows slower than any
polynomial, we see that $g$ can be bounded as
\begin{displaymath}
    |g(u,v)|\le \frac{C_{2} }{|\log( u\overline u)\log( v\overline
      v)|^{1+\varepsilon } u\overline u v\overline v},  
\end{displaymath}
with $C_{2}$ and $\varepsilon $ positive. Since the differential form
\begin{displaymath}
  \frac{\dd u\land \dd
\overline u\land \dd v\land \dd
\overline v}{|\log( u\overline u)\log( v\overline
      v)|^{1+\varepsilon } u\overline u v\overline v}
\end{displaymath}
is locally integrable, we deduce that $\partial\overline \partial
f_{n,m}\land \varphi$ is locally integrable. 

Every product between a smooth form and any of the forms $f_{n,m}$,
$\partial f_{n,m}$,  
$\overline \partial
f_{n,m}$, and $\partial\overline \partial f_{n,m}$, will satisfy
growth estimates 
not worse than the one satisfied by 
$\partial\overline \partial f_{n,m}\land \varphi$, except the
product $\partial\overline \partial
f_{n,m}\land \partial\overline \partial f_{n,m}$. Since this last
product is zero we conclude \eqref{item:2}.

The statement \eqref{item:3} follows from a direct computation.
\end{proof}

\nnpar{The Mumford-Lear extension of the line bundle of
  Jacobi forms to $E(N)$.} 
We now denote by $D=E(N)\setminus E^{0}(N)$. This is a normal
crossings divisor. Let $\Sigma $ be the set of double points of
$D$ and put $D^{0}=D\setminus \Sigma $ for the smooth part of $D$. Let
$H$ be the divisor of $E(N)$ defined as the image of the zero section 
$X(N) \to E(N)$.

Consider the divisor on $E(N)$ given by
\begin{equation}\label{eq:7}
  C=8H+\sum\limits_{j=1}^{p_{N}}
    \sum\limits_{\nu=0}^{N-1}\bigg(N-4\nu+\frac{4\nu ^{2}}{N}\bigg)\Theta_{j,\nu},
\end{equation}
  Choose
  a smooth hermitian metric $\|\cdot\|'$ on $\mathcal{O}(C)$ and let $s$ be a
  section of $\mathcal{O}(C)$ with $\dv s=C$.

\begin{proposition}\label{prop:2} The hermitian line bundle $\overline
  L=(L_{4,4,N},\|\cdot\|)$  satisfies the following properties.
  \begin{enumerate}[(i)]
  \item \label{item:4} The restriction of the metric $\|\cdot\|$ to
    $E^{0}(N)$ is
  smooth. Moreover the divisor of the restriction of $\theta
  _{1,1}^{8}$ to $E^{0}(N)$ is $8H$. Therefore, there is a unique
  isomorphism $\alpha \colon L_{4,4,N}\to
  \mathcal{O}(C)\mid_{E^{0}(N)}$ that sends $\theta _{1,1}^{8} $ to
  $s$.
\item \label{item:5} Each point $p$ belonging to only one component
  $\Theta _{j,\nu
  }$ has a neighborhood $V$ on which
  \begin{displaymath}
    \log\|\theta _{1,1}^{8}\|^{2}=
    \log{\|s\|'}^{2}+\varphi_{1},
  \end{displaymath}
  where $\varphi_{1}$ is a pre-log-log along $D^{0}$.
\item \label{item:6} On the affine coordinate chart $W^{0}_{j,\nu }$
  defined on Section \ref{sec:ellsurface}, we can write
  \begin{displaymath}
    \log\|\theta _{1,1}^{8}\|^{2}=
    \log{\|s\|'}^{2}+\varphi_{2}-\frac{4}{N}
    \frac{\log(u_{\nu}\bar{u}_{\nu})\log(v_{\nu}\bar{v}_{\nu})}
    {\log(u_{\nu}\bar{u}_{\nu})+\log(v_{\nu}\bar{v}_{\nu})},  
  \end{displaymath}
  where $\varphi_{2}$ is pre-log-log along $D$.
  \end{enumerate}
  In consequence, if we denote also by $\|\cdot\|$ the singular metric
  on $\mathcal{O}(C)$ induced by $\alpha $ and $\|\cdot\|$, then the
  5-tuple $(1,\mathcal{O}(C),\Sigma 
  ,\|\cdot\|,\alpha )$ is a Mumford-Lear 
  extension of the hermitian line bundle $\overline
  L$ to $E(N)$ and the divisor of $\theta_{1,1}^{8}$
  on the universal elliptic surface $E(N)$ is given by
  \begin{align*}
    \dv_{E(N)}(\theta_{1,1}^{8})=C.
  \end{align*}
\end{proposition}
\begin{proof} The metric $\|\cdot\|$ on $L_{4,4,N}$ over the open
  subset $E^{0}(N)$ is induced by a
  smooth metric on the trivial line bundle over
  $\mathbb{H}\times\mathbb{C}$. Since $N\ge 3$, the map
  $\mathbb{H}\times\mathbb{C}\to E^{0}(N)$ is \'etale. Hence, the metric
  $\|\cdot\|$ 
  is smooth on $E^{0}(N)$. Therefore,
  the components of $\dv_{E(N)}(\theta _{1,1}^{8})$ that meet
  the open subset $E^{0}(N)$ come from the theta function. Is well
  known that, for fixed $\tau\in\mathbb{H}$, the zeros of the
  Riemann theta 
  function $\theta_ {1,1}(\tau,z)$ are located at
  $z\in\mathbb{Z}\tau\oplus\mathbb 
  {Z}$; all the zeros are simple. This proves that the restriction of
  $\dv_{E(N)}(\theta _{1,1}^{8})$ to $E^{0}(N)$ is given by
  $8H$. This finishes the proof of \eqref{item:4}.

  By the normality of the group $\Gamma(N)$ in $\mathrm{SL}_{2}
  (\mathbb{Z})$, in order to prove that $\overline L$ admits a
  Mumford-Lear extension  and compute the divisor
  $\dv_{E(N)}(\theta_{1,1}^{8})$, it suffices to work over the 
  cusp $P_{1}=  [\infty]$.

  Consider the open affine chart $W=W_{1,\nu}^0$. By Lemma \ref{metriclocal}, 
  \begin{align*}
    &\log\big(\Vert
    \theta _{1,1}^{8}(\tau,z)\Vert^{2}\big)\Big\vert_{W_{\nu}^0}=
    \log\big(\vert \theta _{1,1}^{8}(\tau,z)
    \vert^{2}\big)\Big\vert_{W} \\
    &\hspace*{5mm}+\frac{4}{N}\bigg((\nu+1)^{2}\log(u_{\nu}\bar{u}_{\nu})+\nu^
    {2}\log(v_{\nu}\bar{v}_{\nu})-\frac{\log(u_{\nu}\bar{u}_{\nu})
      \log(v_{\nu}\bar{v}_
      {\nu})}{\log(u_{\nu}\bar{u}_{\nu})+\log(v_{\nu}\bar{v}_{\nu})}\bigg) \\
    &\hspace*{5mm}+4\log\bigg(-\frac{N}{4\pi}\big(\log(u_{\nu}\bar{u}_{\nu})+\log
    (v_{\nu}\bar{v}_{\nu})\big)\bigg).
  \end{align*}
  
  We first study the term $\log|\theta
  _{1,1}^{8}|^{2}$.  For this, we rewrite expression \eqref{theta} defining
  $\theta_{1,1}$ in 
  terms of the local coordinates $u_{\nu},v_{\nu}$. Using formulas
  \eqref{local}, we obtain
  \begin{align*}
    \theta_{1,1}(\tau,z)&=\sum\limits_{n\in\mathbb{Z}}\mathrm{e}^{\pi i(n+1/2)}q_
    {N}^{N/2(n+1/2)^{2}}\zeta^{n+1/2)} \\[1mm]
    &=\sum\limits_{n\in\mathbb{Z}}\mathrm{e}^{\pi(n+1/2)}u_{\nu}^{N/2(n+1/2)^{2}+
      (\nu+1)(n+1/2)}v_{\nu}^{N/2(n+1/2)^{2}+\nu(n+1/2)}.
  \end{align*}
  Since the vertical component $\Theta_{1,\nu}$ is characterized by the equation
  $v_{\nu}=0$, the multiplicity of $\theta_ {1,1}$ along
  $\Theta_{1,\nu}$ is given by 
  \begin{displaymath}
    \min\limits_{n\in\mathbb{Z}}\bigg(\frac{N}{2}n^{2}+
    \bigg(\frac{N}{2}+\nu\bigg)n+ 
    \frac{N}{8}+\frac{\nu}{2}\bigg).
  \end{displaymath}
  For a real number $x$ we write $\lfloor x \rfloor$ for the bigger
  integer smaller or equal to $x$ and $\epsilon (x)=x-\lfloor x
  \rfloor$. Then one easily checks that 
  \begin{multline*}
    \min\limits_{n\in\mathbb{Z}}\bigg(\frac{N}{2}n^{2}+
    \bigg(\frac{N}{2}+\nu\bigg)n+ 
    \frac{N}{8}+\frac{\nu}{2}\bigg)\\=
    \frac{N}{2}\left(\epsilon^{2} \left(-\frac{\nu }{N}\right)-
      \epsilon \left(-\frac{\nu
        }{N}\right)\right)+\frac{N}{8}-\frac{\nu ^{2}}{2N}.
  \end{multline*}
  Note that this quantity depends on the value of $\nu $ and not just
  on the residue class of $\nu $ modulo $N$. This is because
  $\theta_{1,1}(\tau,z)$ is a multi-valued function on $E^{0}(N)$.
  
  Similarly, the multiplicity of $\theta_ {1,1}$ along
  $\Theta_{1,\nu+1}$ is given by
  \begin{displaymath}
    \frac{N}{2}\left(\epsilon^{2} \left(-\frac{\nu+1 }{N}\right)-
      \epsilon \left(-\frac{\nu+1 }{N}\right)\right)+\frac{N}{8}
    -\frac{(\nu+1) ^{2}}{2N}.
  \end{displaymath}
  Therefore, on $W\setminus H$, we can write
  \begin{align*}
    \log&|\theta_{1,1}^{8}|^{2}=\\
    &\left(
    4N\left(\epsilon^{2} \left(-\frac{\nu+1 }{N}\right)-
      \epsilon \left(-\frac{\nu+1 }{N}\right)\right)+N-\frac{4(\nu+1)
      ^{2}}{N}\right)\log u_{\nu 
  } \overline u_{\nu 
    }+\\&
    \left(
    4N\left(\epsilon^{2} \left(-\frac{\nu }{N}\right)-
      \epsilon \left(-\frac{\nu }{N}\right)\right)+N-\frac{4\nu
      ^{2}}{N}\right)\log v_{\nu
  } \overline v_{\nu 
    }+
\varphi_{3}, 
  \end{align*}
  where $\varphi_{3}$ is a smooth function.

  We next consider the remaining terms of the expresion of $\log\|\theta
  _{1,1}^{8}\|^{2}$. The term
  \begin{displaymath}
    4\log\bigg(-\frac{N}{4\pi}\big(\log(u_{\nu}\bar{u}_{\nu})+\log
    (v_{\nu}\bar{v}_{\nu})\big)\bigg)
  \end{displaymath}
  is pre-log-log along $D$.

The terms $\frac{4\nu ^{2}}{N}\log (v_{\nu}\bar{v}_{\nu})$
and $\frac{4(\nu+1) ^{2}}{N}\log (u_{\nu}\bar{u}_{\nu})$ add
$4\nu ^{2}/N$ and $4(\nu+1) ^{2}/N$ to the multiplicity of the
components $\Theta_{1,\nu}$ and $\Theta_{1,\nu+1}$
respectively. Summing up, we obtain that
\begin{align*}
  \log\big(\Vert
    \theta _{1,1}^{8}&(\tau,z)\Vert^{2}\big)\Big\vert_{W\setminus H}=\\&
\left(4N\left(\epsilon^{2} \left(-\frac{\nu }{N}\right)-
      \epsilon \left(-\frac{\nu }{N}\right)\right)+N\right)\log v_{\nu
  } \overline v_{\nu 
    }+\\&
\left(
    4N\left(\epsilon^{2} \left(-\frac{\nu+1 }{N}\right)-
      \epsilon \left(-\frac{\nu+1 }{N}\right)\right)+N\right)\log u_{\nu
  } \overline u_{\nu 
    }-\\&\frac{4}{N}\bigg(\frac{\log(u_{\nu}\bar{u}_{\nu})\log(v_{\nu}\bar{v}_  
      {\nu})}{\log(u_{\nu}\bar{u}_{\nu})+\log(v_{\nu}\bar{v}_{\nu})}\bigg)+
\varphi_{2},
\end{align*}
where $\varphi_{2}$ is pre-log-log along $D$.

In order to finish the proof of \eqref{item:6}, it only remains to
observe that, for $\nu =0,\dots,N$,
\begin{displaymath}
  4N\left(\epsilon^{2} \left(-\frac{\nu }{N}\right)-
      \epsilon \left(-\frac{\nu }{N}\right)\right)=
    \frac{4\nu ^{2}}{N}-4\nu.
\end{displaymath}

Statement \eqref{item:5} follows from \eqref{item:6} and the fact that,
  by Proposition \ref{prop:3}~\eqref{item:1}, the term
  \begin{displaymath}
    \frac{4}{N}\bigg(-\frac{\log(u_{\nu}\bar{u}_{\nu})\log(v_{\nu}\bar{v}_  
      {\nu})}{\log(u_{\nu}\bar{u}_{\nu})+\log(v_{\nu}\bar{v}_{\nu})}\bigg)
\end{displaymath}
is pre-log-log along $D^{0}$. 
\end{proof}

\nnpar{The self-intersection of $C$.} By Proposition \ref{prop:2}, the
Mumford-Lear extension of $\overline 
L$ to $E(N)$ is isomorphic to $\mathcal{O}(C)$. We  next
compute the self-intersection $C\cdot C$.

\begin{proposition}
  The self intersection product $C\cdot C$ is given by
  \begin{displaymath}
    C\cdot C=\frac{16(N^2+1)p_{N}}{3N}.
  \end{displaymath}
  In particular, for $N=4$, we have $p_{4}=6$, hence $C\cdot C=136$.
\end{proposition}
\begin{proof}
Using the adjunction formula (see for instance
\cite{Kramer:jacobi} proof of Proposition 3.2), we obtain
\begin{displaymath}
  H\cdot H=-\frac{Np_{N}}{12}. 
\end{displaymath}
Moreover
\begin{displaymath}
  H\cdot \Theta_{j,\nu }=
  \begin{cases}
    1,&\text{ if }\nu =0,\\
    0,&\text{ otherwise,}
  \end{cases}
\end{displaymath}
and
\begin{displaymath}
  \Theta_{j,\nu }\cdot
  \Theta_{j',\nu' }=
  \begin{cases}
    -2,&\text{ if }j=j',\ \nu =\nu '\\
    1,&\text{ if }j=j', \ \nu\equiv \nu'\pm1 \mod N\\
    0,&\text{ otherwise.}
  \end{cases}
\end{displaymath}
From these intersection products and the explicit description of $C$
in \eqref{eq:7}, we derive the result.
\end{proof}

\nnpar{The b-divisor of the line bundle of Jacobi forms.}

\begin{theorem}\label{thm:1}
  The line bundle $\overline L=(L_{4,4,N},\|\cdot\|)$ admits all
  Mumford-Lear extensions over $E(N)$. Moreover the associated 
  b-divisor is integrable and the equality
  \begin{displaymath}
    \bdv(\theta _{1,1}^{8})\cdot \bdv(\theta
    _{1,1}^{8})=\frac{16 N p_N}{3}
  \end{displaymath}
  holds.
\end{theorem}
\begin{proof}
  Recall that $\Sigma \subset D$ denotes the set of double points of $D$. By
  Proposition \ref{prop:2}\eqref{item:5}, the restriction $\overline
  L|_{E(N)\setminus 
    \Sigma }$ has a pre-log metric along $D\setminus \Sigma
  $. Therefore, if $p\not \in \Sigma $ and $\pi \colon X\to E(N)$ is
  the blow-up of $E(N)$ at $p$ we deduce that $\pi
  ^{\ast}\mathcal{O}(C)$ is a Mumford-Lear extension of $\overline L$
  and that
  \begin{displaymath}
    \dv_{X}(\theta _{1,1}^{8})=\pi ^{\ast}\dv_{E(N)}(\theta _{1,1}^{8}).
  \end{displaymath}
  
  Assume now that $p\in \Sigma $ and $\pi \colon X\to E(N)$ is
  the blow-up of $E(N)$ at $p$. Write $\Sigma _{X}$ for the set of
  double points of the total transform of $D$. Then $\# \Sigma _{X}=\#
  \Sigma +1$, because we can write $\Sigma _{X}=(\Sigma \setminus
  \{p\})\cup \{p_{1},p_{2}\}$, where $\{p_{1},p_{2}\}$ is the
  intersection of the exceptional divisor $E$ of the blow-up with the
  strict transform of $D$.

  Proposition \ref{prop:3}\eqref{item:3}
  and Proposition \ref{prop:2}\eqref{item:6}
  imply that
  \begin{displaymath}
  (N,\pi ^{\ast}\mathcal{O}(NC)\otimes \mathcal{O}(-2E))  
  \end{displaymath}
  is a Mumford-Lear extension of 
  $\overline L$ to $X$ (in this case the co-dimension two set is
  $\Sigma _{X}$, and the isomorphism and the metric are the ones
  induced by $\alpha $ and $\|\cdot\|$). Moreover 
  \begin{displaymath}
    \dv_{X}(\theta _{1,1}^{8})=\pi ^{\ast}\dv_{E(N)}(\theta _{1,1}^{8})-\frac{2}{N} E.
  \end{displaymath}
  A similar
  phenomenon occurs on any smooth surface birational to $E(N)$. To
  describe it we need a little of terminology. Let $\pi\colon X\to
  E(N)$ be a proper birational map with $X$ smooth. A point $p\in X$
  will be called mild if the metric of $\pi ^{\ast}\overline L$ is
  smooth or pre-log-log in a neighborhood of $p$. Put $\Sigma
  _{X}\subset X$ for the set of non mild points. We will say that a
  point $p$ has type $(n,m)$ and multiplicity $\mu $ if there is a
  coordinate neighborhood centered at $p$, with coordinates $(u,v)$
  such that   
  \begin{displaymath}
    \log\|\theta _{1,1}^{8}\|=
    \log\|s\|'+\varphi-\frac{\mu }{nm}
    \frac{\log(u\bar{u})\log(v\bar{v})}
    {n\log(u\bar{u})+m\log(v\bar{v})}.  
  \end{displaymath}
Observe that $E(N)$ has $Np_{N}$ non-mild points, all of type $(1,1)$
and multiplicity $4/N$. 

  Assume that $\Sigma _{X}$ is finite and that $\overline L$ admits a
  Mumford-Lear extension $(e_{X},\mathcal{O}(C_{X}),\Sigma
  _{X},\|\cdot\|,\alpha )$ to $X$. Let $D_{X}$ be the total transform
  of $D$ to $X$. If $\pi \colon X'\to X$ is the
  blow-up at a mild point $p\not \in \Sigma _{X}$, then $\Sigma
  _{X'}=\pi ^{-1}\Sigma _{X}$ is finite and $(e_{X},\pi
  ^{\ast}\mathcal{O}(C_{X}))$ is a Mumford-Lear extension of
  $\overline L$ to $X'$. In particular,
  \begin{displaymath}
    \dv_{X'}(\theta _{1,1}^{8})={\pi'} ^{\ast}\dv_{X}(\theta _{1,1}^{8}).
  \end{displaymath}
  Let now $\pi \colon X'\to X$ be the
  blow-up of $X$ at a point $p\in \Sigma _{X}$, with type $(n,m)$ and
  multiplicity $a/b$, with $a,b$ integers. Then, by Proposition
  \ref{prop:3}\eqref{item:3}, $\Sigma _{X'}=(\Sigma_{X} \setminus
  \{p\})\cup \{p_{1},p_{2}\}$, where $\{p_{1},p_{2}\}$ is the
  intersection of the exceptional divisor $E_{X'}$ of the blow-up with the
  strict transform of $D_{X}$. Moreover,
  \begin{displaymath}
    (bnm(n+m)e_{X},\pi
  ^{\ast}\mathcal{O}(bnm(n+m)C_{X})\otimes (-a E_{X'}))
  \end{displaymath}
  is a Mumford-Lear extension of $\overline L$ to $X'$. Hence
  \begin{displaymath}
    \dv_{X'}(\theta
    _{1,1}^{8})=\pi ^{\ast}\dv_{X}(\theta _{1,1}^{8})-\frac{a
    }{bnm(n+m)}E_{X'}. 
  \end{displaymath}
  Note also that the singular
  point $p$ gives rise to two points in $\Sigma _{X'}$, both of
  multiplicity $a/b$, one of type $(n+m,m)$ and the other of type
  $(n,n+m)$. Since the self-intersection of the exceptional
  divisor $E_{X'}$ is $-1$, we deduce that
  \begin{equation}
    \label{eq:8}
    \dv_{X'}(\theta
    _{1,1}^{8})^2={\pi'} ^{\ast}\dv_{X}(\theta _{1,1}^{8})^2-\frac{a^{2}
    }{b^{2}n^{2}m^{2}(n+m)^{2}}.
  \end{equation}

 Since the elements of $\Bir'(E(N))$ can be obtained by
  successive blow-ups at points, for all $X\in
  \Bir'(E(N))$, the set $\Sigma _{X}$ is finite and $\overline L$
  admits a Mumford-Lear extensions to $X$. Hence $\overline L$ admits
  all Mumford-Lear extensions over $E(N)$.

  From the previous discussion, it is clear that, to study the b-divisor
  $\bdv(\theta _{1,1}^{8})$, we can forget the blow-ups at mild
  points and concentrate on blow-ups along non-mild points.

  Consider the labeled binary tree with root labeled by $(1,1)$ and
  such that, if a node is labeled $(n,m)$, the two child nodes are
  labeled $(n+m,m)$ and $(n,n+m)$. Then the labels of the tree are in
  bijection with the set of ordered pairs of co-prime positive
  integers. This tree also describes the type of the non mild points that
  appear by successive blow-ups starting with a point of type $(1,1)$.  

  By equation \eqref{eq:8} and this description of the singular points
  that appear in the tower of blow-ups, we deduce that the b-divisor
  $\bdv(\theta _{1,1}^{8})$ is integrable if and only if the series
  \begin{displaymath}
    \sum_{\substack{n>0,\, m>0\\(n,m)=1}}\frac{1}{n^{2}m^{2}(n+m)^{2}}
  \end{displaymath}
  is absolutely convergent. Since this is the case, we conclude that
  the b-divisor
  $\bdv(\theta _{1,1}^{8})$ is integrable. Moreover, since $X(N)$ has
  $p_{N}$ cusps and over each cusp $E(N)$ has $N$ points of type
  $(1,1)$ and multiplicity $4/N$, we deduce from equation \eqref{eq:8}  
  \begin{displaymath}
    \bdv(\theta _{1,1}^{8})^{2}=C\cdot C-\frac{4^{2}Np_{N}}{N^{2}}
    \sum_{\substack{n>0,\,
        m>0\\(n,m)=1}}\frac{1}{n^{2}m^{2}(n+m)^{2}}.
  \end{displaymath}
  Now we compute
  \begin{multline*}
    \sum_{\substack{n>0,\,
        m>0\\(n,m)=1}}\frac{1}{n^{2}m^{2}(n+m)^{2}}=
    \frac{\sum_{n>0,\,
        m>0}\frac{1}{n^{2}m^{2}(n+m)^{2}}}{\sum_{k>0}\frac{1}{k^{6}}}
    \\=\frac{\zeta (2,2;2)}{\zeta (6)}=\frac{\frac{1}{3}\zeta (6)}{\zeta
      (6)}=\frac{1}{3}, 
  \end{multline*}
  where $\zeta (2,2;2)$ is the Tornheim zeta function that is computed in
  \cite{Tornheim:hds}.

  Therefore
  \begin{displaymath}
    \bdv(\theta _{1,1}^{8})^{2}=C\cdot
    C-\frac{16p_{N}}{3N}=\frac{16(N^{2}+1)}{3N}p_{N}-\frac{16}{3N}p_{N}
    =\frac{16Np_{N}}{3}
  \end{displaymath}
  concluding the proof of the theorem.
\end{proof}

\begin{remark}
  We can rewrite the formula in Theorem as
  \begin{displaymath}
    \bdv(\theta _{1,1}^{8})^{2}=4\cdot4\cdot[\mathrm{PSL}_{2}(\mathbb{Z}):\Gamma 
    (N)]\frac{\zeta(2,2;2)}{\zeta(6)}.
  \end{displaymath}
  Thus this degree can be interpreted as the product of the weight of
  the Jacobi form,  
  its index, the index of the subgroup $\Gamma 
    (N)$ in $\mathrm{PSL}_{2}(\mathbb{Z})$ and a multiple zeta value.
\end{remark}


\section{Interpretation and open questions}
\label{sec:interpretation}

In the previous section we have seen that, when taking into account
the invariant metric, the natural way to extend
the Cartier divisor $\dv(\theta _{1,1}^{8})$ associated to the line
bundle of Jacobi forms, from the universal family of elliptic curves
to a compactification of it, is not as a Cartier divisor, but as a
$\QQ$-b-divisor. In particular, this implies that we can not restrict
ourselves to a single toroidal compactification, but we have to consider the
whole tower of toroidal compactifications. Considering purely the
arithmetic definition of Jacobi forms, this fact was already observed
by the third author \cite[Remark 2.19]{Kramer:atJfhd}.

In this section we will give further evidence that
$\bdv(\theta _{1,1}^{8})$ is the natural extension of $\dv(\theta
_{1,1}^{8})$ by showing that if satisfies direct generalizations of
classical theorems on hermitial line bundles. We will also 
state some open problems and future lines of research.

\nnpar{A Hilbert-Samuel formula.}
First we observe that $\bdv(\theta ^{8}_{1,1})^2$ satisfies a
Hilbert-Samuel type formula.

\begin{theorem}\label{thm:2}
  For each $N\ge 3$, the equality
  \begin{displaymath}
    \bdv(\theta ^{8}_{1,1})^2=\lim_{\ell\to \infty}\frac{\dim
      J_{4\ell,4\ell}(\Gamma (N))}{\ell^{2}/2!}
  \end{displaymath}
  holds. 
\end{theorem}
\begin{proof}
  By Remark \ref{rem:1} and Theorem \ref{thm:1} we have
  \begin{displaymath}
    \lim_{\ell\to \infty}\frac{\dim
      J_{4\ell,4\ell}\big(\Gamma(N)\big)}{\ell^{2}/2!}=\lim_{\ell\to
      \infty}\frac{\frac{8Np_N}{3}\ell^{2}+o\big(\ell^{2}\big)}{\ell^{2}/2!}=\frac{16Np_N}{3}=\bdv(\theta
    ^{8}_{1,1})^2. 
  \end{displaymath}
\end{proof}

\nnpar{Chern-Weil theory.}
The second task is to show that the self intersection product
in the sense of b-divisors is compatible with Chern-Weil theory. We write
\begin{displaymath}
  c_{1}(L_{4,4,N},\|\cdot\|)=\frac{1}{2\pi i}\partial \bar \partial
  \log\|\theta _{1,1}^{8}\|^2. 
\end{displaymath}

\begin{theorem}\label{thm:3}
  For each $N\ge 3$, the equality
  \begin{displaymath}
    \bdv(\theta ^{8}_{1,1})^2=\int _{E(N)}c_{1}(L_{4,4,N},\|\cdot\|)^{\land 2}
  \end{displaymath}
  holds. 
\end{theorem}
\begin{proof}
  By propositions \ref{prop:3} and \ref{prop:2}, we know that the
  integral in the right hand side exists and is finite. 

  Let $C$ be the divisor on Proposition \ref{prop:2} and choose a
  pre-log metric $\|\cdot\|'$ on $\mathcal{O}(C)$ such that, each double
  point $p_{j,\nu }=\Theta _{j,\nu }\cap \Theta _{j,\nu+1 }\in W^{0}_{j,\nu }$
  has a neighborhood in which
  \begin{equation}\label{eq:9}
    \log\|\theta _{1,1}^{8}\|^{2}=
    \log{\|\theta _{1,1}^{8}\|'}^{2}-\frac{4}{N}
    \frac{\log (u_{\nu }\bar u_{\nu })\log (v_{\nu }\bar v_{\nu })}
    {\log (u_{\nu }\bar u_{\nu })+\log (v_{\nu }\bar v_{\nu })}.
  \end{equation}
  
  Write $\omega =c_{1}(L_{4,4,N},\|\cdot\|)$, $\omega
  '=c_{1}(\mathcal{O}(C),\|\cdot\|')$ and
  \begin{displaymath}
    f=\log\|\theta _{1,1}^{8}\|^{2}-
    \log{\|\theta _{1,1}^{8}\|'}^{2}.
  \end{displaymath}
  Thus
  \begin{displaymath}
    \omega =\omega '+\frac{1}{2\pi i}\partial\bar \partial f.
  \end{displaymath}
  
  Since Chern-Weil theory can be extended to pre-log singularities
  (\cite{Mumford:Hptncc}, \cite{BurgosKramerKuehn:accavb}), the equality
  \begin{displaymath}
    \int_{E(N)}{\omega '}^{\land 2}=C\cdot C
  \end{displaymath}
  holds.
  Since
  \begin{displaymath}
    \int_{E(N)}\omega ^{\land 2}= \int_{E(N)}{\omega '}^{\land 2}-
    \int_{E(N)}\dd (\frac{2}{2\pi i}\partial f\land \omega '+
    \frac{1}{(2\pi i)^{2}}\partial 
    f\land \partial\bar \partial f),
  \end{displaymath}
  we are led to compute the second integral of the right hand side of
  the previous equation. Note that the minus sign in the above formula
  comes from the fact that $\dd \partial=-\partial\bar \partial$.
  Since pre-log-log forms have no residue, in
  order to compute this integral we can focus on the double points
  $p_{j,\nu }$, $j=1,\dots, p_{N}$, $\nu =0,\dots,N-1$ of
  $D$. For each point $p_{j,\nu }$ and $0<\varepsilon <1/e$, let
  $V_{j,\nu ,\varepsilon }$ be the 
  poly-cylinder
  \begin{displaymath}
    V_{j,\nu ,\varepsilon }=\{(u_{\nu },v_{\nu })\in W^{0}_{j,\nu }\mid
    |u_{\nu }|\le \varepsilon ,\ |v_{\nu }|\le \varepsilon\}. 
  \end{displaymath}
  
  Then, by Stokes theorem,
  \begin{multline*}
    -\int_{E(N)}\dd (\frac{2}{2\pi i}\partial f\land \omega '+
    \frac{1}{(2\pi i)^{2}}\partial 
    f\land \partial\bar \partial f)=\\
    \sum_{j=1}^{p_{N}}\sum_{\nu =0}^{N-1}\lim _{\varepsilon \to
      0}\int_{\partial V_{j,\nu ,\varepsilon }}\frac{2}{2\pi i}\partial
    f\land \omega '+ \frac{1}{(2\pi i)^{2}}\partial 
    f\land \partial\bar \partial f.
  \end{multline*}
  Using that $\omega '$ is a pre-log-log form and equation \eqref{eq:4},
  it is easy to see that
  \begin{displaymath}
    \lim_{\varepsilon \to 0} \int_{\partial V_{j,\nu ,\varepsilon
      }}\frac{2}{2\pi i}\partial 
    f\land \omega '=0.
  \end{displaymath}
  
  For shorthand, write $(u,v)$ for the coordinates $(u_{\nu },v_{\nu })$
  of $W^{0}_{j,\nu }$. We decompose $V_{j,\nu
    ,\varepsilon}=A_{\varepsilon }\cup B_{\varepsilon }$, where 
  \begin{align*}
    A_{\varepsilon }&=\{(u,v)\in W^{0}_{j,\nu }\mid
    |u|\le \varepsilon ,\ |v|= \varepsilon\},\\
    B_{\varepsilon }&=\{(u,v)\in W^{0}_{j,\nu }\mid
    |u|= \varepsilon ,\ |v|\le \varepsilon\}.
  \end{align*}
  Using equations \eqref{eq:9} and \eqref{eq:10} and taking care of the
  canonical orientation of a complex manifold, we see that
  \begin{displaymath}
    \int_{A_{\varepsilon }}\frac{1}{(2\pi i)^{2}} \partial
    f\land \partial\bar \partial f = \frac{16}{N^{2}}
    \int_{0}^{\varepsilon }\frac{2(\log(\varepsilon
      ^{2}))^{2}\log(r^{2})2r\dd r}{(\log(r^{2})+\log(\varepsilon
      ^{2}))^{4}r^{2}}=\frac{-16}{6N^{2}}.
  \end{displaymath}
  Similarly
  \begin{displaymath}
    \int_{B_{\varepsilon }}\frac{1}{(2\pi i)^{2}} \partial
    f\land \partial\bar \partial f =\frac{16}{N^{2}}
    \int_{0}^{\varepsilon }\frac{2(\log(\varepsilon
      ^{2}))^{2}\log(r^{2})2r\dd r}{(\log(r^{2})+\log(\varepsilon
      ^{2}))^{4}r^{2}}=\frac{-16}{6N^{2}}.
  \end{displaymath}
  Hence
  \begin{displaymath}
    \lim _{\varepsilon \to
      0}\int_{\partial V_{j,\nu ,\varepsilon }}\frac{2}{2\pi i}\partial
    f\land \omega '+ \frac{1}{(2\pi i)^{2}}\partial 
    f\land \partial\bar \partial f=\frac{-16}{3N^{2}}
  \end{displaymath}
  
  Therefore
  \begin{displaymath}
    \int _{E(N)}c_{1}(L_{4,4,N},\|\cdot\|)^{\land 2}=C\cdot C-\frac{16
      p_{N}}{3N}= \bdv(\theta ^{8}_{1,1})^2.
  \end{displaymath}
\end{proof}

\begin{remark}\label{rem:2}
  Recall the function
  \begin{displaymath}
    f_{1,1}(x,y)=\frac{\log(x\overline x)\log(y\overline y)}
    {\log(x\overline x)+\log(y\overline y)}.
  \end{displaymath}
  The heart of the proof of Theorem \ref{thm:3} is the relation
  \begin{displaymath}
    -\Res_{(0,0)}\Big(\frac{1}{(2\pi i)^{2}}\partial
      f_{1,1}\land \partial\bar \partial
      f_{1,1}\Big)=\frac{1}{3}=\sum_{\substack{n>0,\,
        m>0\\(n,m)=1}}\frac{1}{n^{2}m^{2}(n+m)^{2}}
  \end{displaymath}
  between the residue at $(0,0)$ of the differential form
  $\frac{1}{(2\pi i)^{2}}\partial f_{1,1}\land \partial\bar \partial
  f_{1,1}$ and the harmonic double value $\zeta (2,2;2)/\zeta (6)$. This gives
  us a geometric interpretation of this harmonic double value. 
\end{remark}

\nnpar{Intersections with curves.} Similarly, we also note that the
intersection of $\bdv(\theta _{1,1}^{8})$ with a curve can also be
computed using the differential form $c_{1}(L_{4,4,N},\|\cdot\|)$.

To a curve $C$ contained in $E(N)$, we associate the b-divisor that,
on each $X\in \Bir'(E(N))$ consist on the strict transform of $C$ on
$X$. We will denote this divisor by $\bdv(C)$. Note that this
b-divisor is not integrable because by taking successive blow-ups in
points of $C$, the strict transform of $C$ has self-intersection
more and more negative. Assume that $C$ is irrecucible and
is not contained in $D=E(N)\setminus E^{0}(N)$. Then the product
$\bdv(\theta _{1,1}^{8})\cdot \bdv(C)$ is well defined because after
a finite number of blow-ups on the double points of $D$ and of its total
transforms, the strict transform of $C$ will not meet any double point
of the total transform of $D$.
\begin{theorem}\label{thm:5}
  The equality
  \begin{displaymath}
    \bdv(\theta _{1,1}^{8})\cdot \bdv(C)=
    \int_{C}c_{1}(L_{4,4,N},\|\cdot\|)
  \end{displaymath}
  holds.
\end{theorem}
\begin{proof}
  Let $X\to E(N)$ be a birational map obtained by successive blow-ups
  on double points of $D$ and of its total transforms and such that
  the strict transform of $C$ in $X$, that we denote by $C_{X}$, does
  not meet any double point of the total transform of $D$ to $X$. Then
  \begin{displaymath}
    \bdv(\theta _{1,1}^{8})\cdot \bdv(C)=\dv_{X}(\theta
    _{1,1}^{8})\cdot C_{X}. 
  \end{displaymath}
  Let $(e,\mathcal{L},S,\alpha ,\|\cdot\|)$ be a Mumford-Lear
  extension of $\overline L=(L_{4,4,N},\|\cdot\|)$ to $X$. Denote by $s
  =\alpha (\theta ^{8e}_{1,1})$ the rational section of
  $\mathcal{L}$ determined by $\theta _{1,1}^{8}$. Since the metric
  $\|\cdot\|$ is pre-log on $X\setminus S$, we deduce that
  \begin{displaymath}
    \dv_{X}(\theta
    _{1,1}^{8})\cdot C_{X}=\frac{1}{e}\dv(s)\cdot C_{X}=
    \frac{1}{e} \int_{C}c_{1}(\mathcal{L},\|\cdot\|)=
    \int_{C}c_{1}(L_{4,4,N},\|\cdot\|).
  \end{displaymath}
\end{proof}

\nnpar{A toric analogue of the singular metric.}
We now give an interpretation of the harmonic double value
$\zeta (2,2;2)/\zeta (6)$ in 
terms of toric varieties and the volume of a convex surface. 

Consider the projective plane $\PP^{2}$ with projective coordinates
$(x_{0}:x_{1}:x_{2})$ and the canonical line bundle
$\mathcal{O}(1)$. On this line bundle we can put the canonical metric
given by
\begin{displaymath}
  \|x_{0}\|_{\can}=\frac{|x_{0}|}{\max(|x_{0}|,|x_{1}|,|x_{2}|)}.
\end{displaymath}
This metric is continuous.
We have an open immersion $(\CC^{\ast})^{2}\hookrightarrow \PP^{2}$ that
sends the point $(z_{1},z_{2})$ to $(1:z_{1}:z_{2})$. We define the
valuation map $\val\colon (\CC^{\ast})^{2}\to \RR^{2}$ by
\begin{displaymath}
  \val(z_{1},z_{2})=(-\log|z_{1}|,-\log|z_{2}|)
\end{displaymath}
The function $\log(\|x_{0}\|_{\can})$ is constant along the fibers of
$\val$. Thus there exist a function $\Psi _{\can}\colon \RR^{2}\to \RR$
such that
\begin{displaymath}
  \log\|x_{0}(p)\|_{\can}=\Psi _{\can}(\val(p)).
\end{displaymath}
This function is explicitly given by
\begin{displaymath}
  \Psi _{\can}(u,v)=\min(0,u,v).
\end{displaymath}
The projective plane $\PP^{2}$ is a toric variety with the action of
$(\CC^{\ast})^{2}$ given by
\begin{displaymath}
  (\lambda ,\mu )(x_{0}:x_{1}:x_{2})=(x_{0}:\lambda x_{1}:\mu x_{2}). 
\end{displaymath}

The theory of toric varieties tells us that the polytope associated to
$\dv(x_{0})$ is the stability set of $\Psi _{can}$:
\begin{multline*}
  \Delta =\{ x\in (\RR^{2})^{\vee}\mid x(u,v)\ge \Psi _{\can}(u,v),\
  \forall (u,v)\in \RR^{2}\}\\=\conv((0,0),(1,0),(0,1)).
\end{multline*}
Moreover
\begin{displaymath}
  \dv(x_{0})^{2}=2\Vol(\Delta ))=1,
\end{displaymath}
where the volume is computed with respect to the Haar measure that
gives $\ZZ^{2}$ covolume 1.

Now we want to modify the canonical metric to introduce a singularity
of the same type as the singularity of the translation invariant metric on
the line bundle of Jacobi forms at the double points. We define the
metric $\|\cdot\|_{\sing}$ by
\begin{multline*}
  \log\|x_{0}\|_{\sing}=\\
  \begin{cases}
    -\frac{\log(|x_{1}/x_{0}|)\log(|x_{2}/x_{0}|)}
    {\log(|x_{1}/x_{0}|)+\log(|x_{2}/x_{0}|)} &\text{ if }
    |x_{0}|\ge \max(|x_{1}|,|x_{2}|),\\
    -\max(\log(|x_{1}/x_{0}|),\log(|x_{2}/x_{0}|))&
    \text{ otherwise.}
  \end{cases}
\end{multline*}
As before, the function $\log\|x_{0}\|_{\sing}$ is constant along the
fibers of $\val$ and defines a function $\Psi _{\sing}\colon
\RR^{2}\to \RR$ that is given explicitly by
\begin{displaymath}
  \Psi _{\sing}(u,v)=
  \begin{cases}
    \frac{uv}{u+v},&\text{ if }u,v\ge 0,\\
    u, &\text{ if } u\le \min(0,v),\\
    v, &\text{ if } v\le \min(0,u).
  \end{cases}
\end{displaymath}
The function $\Psi _{\sing}$ is conic but is not piecewise
linear. Assume that we can extend 
the theory of toric varieties to toric b-divisors. Then to $\Psi
_{\sing}$ we would associate the convex figure
\begin{multline*}
  \Delta_{\sing} =\{ x\in (\RR^{2})^{\vee}\mid x(u,v)\ge \Psi _{\sing}(u,v),\
  \forall (u,v)\in \RR^{2}\},
\end{multline*}
and we should obtain
\begin{equation}
  \label{eq:11}
  \bdv(x_{0},\|\cdot\|_{\sing})^{2}=2\Vol(\Delta _{\sing}).
\end{equation}
We see that this is indeed the case.

\begin{theorem}\label{thm:4}
  The equation \eqref{eq:11} holds.
\end{theorem}
\begin{proof}
  Arguing as in the proof of Theorem \ref{thm:1}, we see that
  \begin{displaymath}
    \dv(x_{0})^{2}-\bdv(x_{0},\|\cdot\|_{\sing})^{2}=\sum_{\substack{n>0,\,
        m>0\\(n,m)=1}}\frac{1}{n^{2}m^{2}(n+m)^{2}}=\frac{1}{3}.
  \end{displaymath}
  The stability set $\Delta _{\sing}$ can be explicitly computed, and
  is given by
  \begin{displaymath}
    \Delta _{\sing}=\{(x,y)\in (\RR^{2})^{\vee}\mid x,y\ge 0,\,x+y\le
    1, \sqrt{x}+\sqrt{y}\ge 1.\}
  \end{displaymath}
  Thus
  \begin{displaymath}
    2\Vol(\Delta )-2\Vol(\Delta _{\sing})=
    2\int_{0}^{1} (1-\sqrt{x})^{2}\dd x=\frac{1}{3}.
  \end{displaymath}
\end{proof}
\begin{remark}\label{rem:3} In fact, since in the theory of toric
  varieties, the blow-ups have a explicit description in terms of
  fans, it is possible to interpret the equation $\zeta (2,2;2)=1/3\zeta (6)$
  is terms of an infinite triangulation of $\Delta \setminus \Delta
  _{\sing}$. 
\end{remark}

\nnpar{Open questions.} In this paper we have examined a particular
example and observed, just by comparing numbers, that several
classical results
should be extendable to b-divisors and singular metrics with a
shape similar to the one of the translation invariant metrics. We are
in the process of investigating the following questions.

\begin{enumerate}
\item Theorem \ref{thm:2} shows that the translation invariant metric
  encodes the asymptotic behavior of the space of Jacobi forms. It is
  possible to define global sections of a b-divisor. We can ask what
  is the exact relationship between the space of Jacobi forms and the
  global sections of the b-divisor $\dv(\theta _{1,1}^{8})$. Moreover,
  once this is settled, we can ask whether there is a Riemann-Roch
  theorem or a Hilbert-Samuel theorem for b-divisors that imply
  directly Theorem \ref{thm:2}.
\item By Theorem \ref{thm:4}, it is clear that much of the theory of
  toric varieties could be extended to toric b-divisors and singular
  metrics on toric varieties.
\item Theorem \ref{thm:3} shows that Chern-Weil theory of singular
  metrics can be useful to study b-divisors. It would be
  interesting to generalize this theorem to higher dimensions. In this
  direction, with R. de Jong and
  D. Holmes, we have shown that the local integrability property
  extends, at least, to the case of toroidal compactifications of families of
  abelian varieties. 
\item The original motivation of this paper is to be able to define
  and study the height of cycles on the universal elliptic curve with
  respect to
  the bundle of Jacobi forms equipped with the translation invariant
  metric, extending the work in \cite{Kramer:atsjf}. First it 
  is clear how to define the height of an algebraic point of $E^{0}(N)$
  and one may wonder whether the new singularities are mild enough so
  that Northcott property is still true. We can
  also define the height of an algebraic curve not contained in the
  divisor $D$. But it is not clear how to define the height of
  $E(N)$. The naive definition of that height would give the value
  $-\infty$
  but it should be possible to extract a meaningful finite number. To
  this end, the study of toric varieties might be useful, because the
  techniques developed in \cite{BurgosPhilipponSombra:agtvmmh} can be
  extended to the singular metrics of Theorem \ref{thm:4}. In this
  case, we obtain that the stability set of the function associated to
  the metric is no longer a polytope but a convex set. In this case
  the regularized height should be defined from the integral along this
  convex set of the roof function, in analogy with \cite[Theorem
  5.2.5]{BurgosPhilipponSombra:agtvmmh}. 
\end{enumerate}


\bibliographystyle{amsplain}
\newcommand{\noopsort}[1]{} \newcommand{\printfirst}[2]{#1}
  \newcommand{\singleletter}[1]{#1} \newcommand{\switchargs}[2]{#2#1}
  \def\cprime{$'$}
\providecommand{\bysame}{\leavevmode\hbox to3em{\hrulefill}\thinspace}
\providecommand{\MR}{\relax\ifhmode\unskip\space\fi MR }
\providecommand{\MRhref}[2]{%
  \href{http://www.ams.org/mathscinet-getitem?mr=#1}{#2}
}
\providecommand{\href}[2]{#2}


\end{document}